\documentclass[a4paper,12pt]{amsart}

\DeclareRobustCommand{\SkipTocEntry}[5]{}

\usepackage{shadow}

\usepackage{accents}
\usepackage{marvosym}

\usepackage[T1]{fontenc}
\usepackage{empheq}
\usepackage{relsize}
\usepackage{amsmath}

\usepackage{bm}

\usepackage{upgreek}
\usepackage{ esint }
\usepackage{color}
\usepackage{comment}
\usepackage{amssymb}
\usepackage{amsfonts}
\usepackage{graphicx}

\usepackage{amscd}

\usepackage{slashed}
\usepackage{pdflscape}
\usepackage{tikz}
\usepackage{enumerate}
\usepackage{multirow}
\usepackage{stmaryrd}
\usepackage{cancel}
\usepackage{t1enc}

\usepackage{amssymb}
\usepackage{amscd}
\usepackage{scalerel,stackengine}

\usepackage{ifthen}

\usepackage{pdfsync}

\usepackage{graphicx} 
\usepackage{amsmath} 
\usepackage{amsfonts}
\usepackage{amssymb}
\usepackage{amsthm}

\usepackage{bbold}
\usepackage[linkcolor=blue]{hyperref}
\usepackage{mathrsfs}
\usepackage[colorinlistoftodos]{todonotes}
\usepackage{ytableau}

\definecolor{blue}{rgb}{.255,.41,.884} 
\definecolor{red}{rgb}{1, 0, 0} 
\definecolor{green}{rgb}{.196,.804,.196} 
\definecolor{yellow}{rgb}{1,.648,0} 
\definecolor{pink}{rgb}{1,0.5,0.5}

\newtheorem{theorem}{Theorem}[section]
\newtheorem{lemma}[theorem]{Lemma}
  
\newtheorem{proposition}[theorem]{Proposition}
\newtheorem{corollary}[theorem]{Corollary}
\newtheorem{conjecture}[theorem]{Conjecture}

\theoremstyle{definition}
\newtheorem{definition}[theorem]{Definition}

\theoremstyle{remark}
\newtheorem{remark}[theorem]{Remark}

\newcommand{\setEnvironmentQed}[2]{
  \AtBeginEnvironment{#1}{%
    \pushQED{\qed}\renewcommand{\qedsymbol}{#2}%
  }
  \AtEndEnvironment{#1}{\popQED}
}

\newcommand{\Conv}{\mathop{\scalebox{1.5}{\raisebox{-0.2ex}{$\ast$}}}}
\setEnvironmentQed{definition}{\ensuremath{\Conv}}

\newcommand{\II}{{\rm  I\hspace{-.2mm}I}}
\newcommand{\IIo}{\hspace{0.4mm}\mathring{\rm{ I\hspace{-.2mm} I}}{\hspace{.0mm}}}

\newcommand{\Fo}{ {\hspace{.6mm}} \mathring{\!{ F}}{\hspace{.2mm}}}

\newcommand{\IIIo}{{\mathring{{\bf\rm I\hspace{-.2mm} I \hspace{-.2mm} I}}{\hspace{.2mm}}}{}}
\newcommand{\IVo}{{\mathring{{\bf\rm I\hspace{-.2mm} V}}{\hspace{.2mm}}}{}}
\newcommand{\Vo}{{\mathring{{\bf\rm V}}}{}}

\newcommand{\otop}{\mathring{\top}}

\newcommand{\ct}{\mathcal{T}}

\newcommand{\cU}{\mathcal{U}}
\newcommand{\cK}{\mathcal{K}}

\newcounter{mnotecount}

\newcommand{\mnotex}[1]
{\protect{\stepcounter{mnotecount}}$^{\mbox{\footnotesize $\bullet$\themnotecount}}$ 
\marginpar{
\raggedright\tiny\em
$\!\!\!\!\!\!\,\bullet$\themnotecount: #1} }

\numberwithin{equation}{section}

\newcommand{\cc}{\boldsymbol{c}}
\newcommand{\bg}{\boldsymbol{g}}

\renewcommand{\=}{\stackrel\Sigma =}

\DeclareMathOperator{\tr}{tr}

\renewcommand\geq{\geqslant}
\renewcommand\leq{\leqslant}

\stackMath
\newcommand\reallywidehat[1]{%
\savestack{\tmpbox}{\stretchto{%
  \scaleto{%
    \scalerel*[\widthof{\ensuremath{#1}}]{\kern-.6pt\bigwedge\kern-.6pt}%
    {\rule[-\textheight/2]{1ex}{\textheight}}
  }{\textheight}%
}{0.5ex}}%
\stackon[1pt]{#1}{\tmpbox}%
}

\newcommand{\csdot}{\hspace{-0.75mm} \cdot \hspace{-0.75mm}}

\newcommand{\bdot }{\mathop{\lower0.33ex\hbox{\LARGE$\cdot$}}}

\definecolor{ao}{rgb}{0.0,0.0,1.0}
\definecolor{forest}{rgb}{0.0,0.3,0.0}
\definecolor{red}{rgb}{0.8, 0.0, 0.0}

\newcommand{\FF}[1]{\mathring{\underline{\overline{\rm{#1}}}}}

\newcommand{\ce}{\mathcal{E}}

\begin{document}


\renewcommand{\today}{}
\title{Conformal hypersurface invariants and Bach-type Boundary Problems
}
%
%
%

\author{ Samuel Blitz${}^\flat$ \&  Rod Gover${}^\sharp$}

\address{${}^\flat$
 Department of Mathematics and Statistics \\
 Masaryk University\\
 Building 08, Kotl\'a\v{r}sk\'a 2 \\
 Brno, CZ 61137} 
   \email{sam.h.blitz@gmail.com}

\address{${}^\sharp$
  Department of Mathematics\\
  The University of Auckland\\
  Private Bag 92019\\
  Auckland 1142\\
  New Zealand,  and\\
  Mathematical Sciences Institute, Australian National University, ACT 
  0200, Australia} \email{r.gover@auckland.ac.nz}
 
\vspace{10pt}


\maketitle

\pagestyle{myheadings} \markboth{S. Blitz \& A. R. Gover}{Conformal hypersurface invariants}

\begin{abstract}
  Using variational considerations, we establish that there exists a
  new symmetric trace-free tensor conformal invariant of hypersurfaces
  embeddings in even dimensional conformal manifolds. This conformal
  invariant completes the family of conformal invariants known as
  conformal fundamental forms.

  The object has important links to global problems.  In the context
  of the even dimensional boundary-value Poincar\'e--Einstein problem,
  the image of the Dirichlet--to--Neumann map is conformally
  invariant. Recent investigations~\cite{BGKW} established that this
  image is the pullback of a particular Riemannian invariant to the
  odd-dimensional boundary. We show here that, in fact, that image
  arises as the restriction of the new conformal invariant constructed
  here.
As a consequence of the proof, we are able to construct several
new global conformal invariants of the boundary. Finally, we use
our variational results to establish that compact Bach-flat manifolds
with umbilic boundary must admit a (formal to all orders)
Poincar\'e--Einstein metric in the conformal class of its interior.

\vspace{0.5cm}


\end{abstract}

\vspace{2cm}

\section{Introduction}

This work concerns the geometric study of conformal {\em
  hypersurfaces}, meaning embedded submanifolds of codimension one in
conformal manifolds of dimension $d$. Hypersurfaces are especially
important, in part because they arise as the boundaries of domains and
manifolds with boundary and in these settings their geometry plays a
pivotal role in many geometric boundary problems. Turning this around
boundary problems can sometimes be used to help recover the geometry
of hypersurfaces; this idea has been effectively applied in number of
recent settings \cite{Graham-renvol,GW-renvol,AGW,will1,BGW1,BGKW} and it will also be exploited here.

While describing the local extrinsic geometry of Riemannian
submanifolds is a classical subject \cite{doCarmo,kobayashinomizu,spivak}, the corresponding
conformal theory has, apart from a few basic elements (see for example \cite{Fialkow}), been developed far more recently and is considerably more subtle and
rich than its Riemannian counterpart. Even obtaining local invariants
effectively is not trivial. For a conceptual and practical approach it
is natural to exploit the conformally invariant Cartan connection
\cite{Cartan} or its associated bundle analogue known as the tractor
connection \cite{BEG}, and the associated calculus.  A
comprehensive basic conformal submanifold theory is developed in
\cite{curry-gover-snell} building on \cite{curry-thesis,stafford}
and this can be used to proliferate local invariants.

A distinguished class of extrinsic hypersurface invariants was
recently found and constructed in a rather different way
\cite{BGW1,Blitz1}. These {\em higher conformal fundamental forms}
$$
\IIo ,\IIo, \cdots ,\FF{k}   \qquad k\leq \left\{ \begin{array}{ll}d-1 & d \,\, \mbox{even}\\
  \frac{d+1}{3} & d \,\, \mbox{odd}
  \end{array} \right. .
$$
are natural trace-free symmetric tensors, each of which is an
extrinsic hypersurface conformal invariant. Here $\IIo$ is simply the
trace-free part of the usual second fundamental form, while $\IIIo$ is
the trace-free part of a conformal invariant due to Fialkow (that also
arises naturally in the tractor picture
\cite{curry-gover-snell,stafford,vyatkin}), but the higher $\FF{k}$
are new objects that probe higher jets of the embedding.  They are
central objects in two ways. First they provide a generating basis of
extrinsic invariants in the sense that any local irreducible natural
hypersurface invariant (with not too high a transverse order---see
below for the meaning of this) can be expressed by partial contraction
polynomials which involve only intrinsic hypersurface invariants and these
fundamental forms \cite{blitz-classifying}.

The second reason is through their relation to Poincar\'{e}-Einstein
manifolds. On a manifold $M$ with boundary $\Sigma:=\partial M$ an
interior Riemannian metric $g_+$ is said to be conformally compact if
takes the form $g_+=r^{-2}g$ where $g$ is a metric on $M$ and $r$ is a
defining function for the boundary -- meaning $\Sigma=r^{-1}(0)$ and
$dr$ is nowhere zero along $\Sigma$. If $g_+$ is Einstein then its
scalar curvature is negative and $|dr|^2_g$ is constant (and, by
scaling $g_+$, is set to 1) along $\Sigma$, then it is said to
be a {\em Poincar\'e-Einstein (PE) metric}. It is well known that
$g_+$ being PE also implies $\IIo=0$, and also that the trace-free
part of the Fialkow tensor $\Fo$ vanishes \cite{Fialkow,GoverAE}, but
in fact each of the $\FF{k}$ must vanish \cite{BGW1}. Thus the higher
conformal fundamental forms provide order-by-order obstructions to the
PE condition.


We henceforth restrict to the case of $d$ even. At order $k=d$
(involving the $d-1$-jet of $(r,g)$) we do not expect local
obstructions to the PE condition, as this is the order at which global
data enters the asymptotics \cite{graham2005dirichlet,BGKW}. However this does
not exclude the existence of a conformal fundamental form at this
order. Indeed in the work \cite{BGKW}  the authors describe a local
curvature invariant of PE manifolds that precisely captures the image
of the Dirichlet-Neumann type map associated with the PE filling
problem. A natural question, first posed in~\cite{BGKW}, is whether this invariant is the
restriction to PE manfolds of a local conformal hypersurface invariant
(that exists on any even conformal manifold with boundary). The
invariant
$$\IVo_{ab} := C_{n(ab)}^\top + H W_{nabn} - \bar{\nabla}^c W_{c(ab)n}^\top$$ of \cite{BGW1} is an invariant of hypersurfaces in conformal
manifolds of dimension $d=4$ that is exactly of this type: when
applied along the boundary of a PE manifold, it recovers the invariant $\operatorname{DN}^{(4)}$ of \cite{BGKW}. Our main result is that a similar
result is available in all even dimensions.
\begin{theorem}\label{main}
For a hypersurface embedded in a conformal manifold $\Sigma \hookrightarrow (M^d,\cc)$ of even dimension $d$ there
is natural conformal extrinsic invariant $\FF{d} \in \Gamma(\odot^2_\circ T^* \Sigma[3-d])$ such that $\FF{d}$ has transverse order $d-1$.
  \end{theorem}

In the above and what follows, the \textit{transverse order} of an
extrinsic invariant is the integer $k$ such that, when the invariant
$I[g]$ is evaluated on a hypersurface embedding $\Sigma
\hookrightarrow (M,g+s^{\ell} h)$ with $s$ a defining function for
$\Sigma$ and for any choices of metric $g$ and rank-2 tensor $h$, we
find that $I[g] = I[g+s^{\ell} h]$ for all $\ell > k$, and $I[g] \neq
I[g + s^k h]$. The transverse order of a differential operator is
similarly defined.
See~\cite{BGW1,BGKW} for more details.

It is
straightforward to verify that the transverse order, weight (as described in Section \ref{sec:tractors}), and
symmetry properties of $\FF{d}$ above are sufficient to uniquely
characterize the leading transverse-order structure of this invariant
(see~\cite{BGW1}).  Moreover this uniqueness is enough to give the
following:
\begin{theorem}\label{DN}
On even-dimensional conformal manifolds $(M^d,\cc)$ with boundary that
admit a Poincar\'e--Einstein structure, the map
\begin{equation}\label{d-map}
  \left(\partial M \hookrightarrow (M, \cc) \right) \mapsto \FF{d}
  \end{equation}
recovers the Poincar\'e--Einstein Dirichlet-to-Neumann map of \cite{BGKW}.
\end{theorem}

This Theorem is a key motivation for our work here.  Fixing a manifold
with boundary $(M,\Sigma=\partial M)$, the Poincar\'e-Einstein filling
problem is the PDE problem of finding a Poincar\'e-Einstein metric
$g_+$ on the interior that is compatible with a given conformal
structure $\cc_\Sigma$ on $\Sigma$. The latter is viewed as the
Dirichlet data for this non-linear problem and the corresponnding
Dirichlet-to-Neumann map, introduced in \cite{graham2005dirichlet} and studied
in, for example \cite{gover2007conformal,wang2019dirichlet,BGKW}, is the map
that extracts the key non-locally determined data in the metric
asymptotics. Thus the map carries information about the global
structure of $(M,\Sigma,\cc_\Sigma)$. It is not obvious that there
should generally be an invariant of conformal hypersurface embeddings
$\Sigma \hookrightarrow (M, \cc) $ that upon restriction recovers this
important quantity, but the Theorem establishes that this is the case.

 In Riemannian geometry the first fundamental form of a hypersurface
 is the induced metric thereon. In this sense for a conformal
 hypersurface $\Sigma$ we may think of the induced conformal metric as the
 first fundamental form, this is not trace-free but naturally its variations are given by sections of $S^2_oT^*\Sigma [2]$. Then $\FF{d}$ is the dual object in that  $\FF{d}$ has the right symmetries
 and weight to pair along the boundary with such variations:
 For any such $h_{ab}$
\begin{equation}\label{pair}
\int_\Sigma h^{ab}\FF{d}_{ab} \, dv_{g_\Sigma}
\end{equation}
is conformally invariant. Furthermore, the divergence $\bar{\nabla}^a \FF{d}_{ab}$ is also conformally invariant, and it appears as an obstruction to a conformal manifold with boundary admitting a Poincar\'e--Einstein metric in the conformal class on its interior $M_+$.
\begin{corollary} \label{divergence-corollary}
Let $\Sigma \hookrightarrow (M,\cc)$ be an even-dimensional conformal manifold with boundary. The divergence
$$\bar{\nabla}^a \FF{d}_{ab} \in \Gamma(T^*\Sigma[1-d])$$
is conformally invariant and is an obstruction to $\cc|_{M_+}$ admitting a Poincar\'e--Einstein metric.
  \end{corollary}
This invariant also plays a role in the Bach-flat (see Equation~(\ref{bach-def} for a definition of the Bach tensor) boundary-value problem in four dimensions. In particular, one finds that the normal-tangential projection of the Bach tensor along the conformal boundary depends on $\bar{\nabla}^a \IVo_{ab}$.

\medskip 

The observation \eqref{pair} above is part of a more general
obervation that weights of the conformal fundamental forms are such
that they yield a family of extrinsic conformally invariant integrals along the conformal boundary:
\begin{proposition}
Let $\Sigma \hookrightarrow (M^d,\cc)$ be a conformal hypersurface embedding with $d$ even. Then, for each $k \in \{2, \ldots, \tfrac{d}{2}\}$, the integral
\begin{align} \label{pair2}
\int_{\Sigma} \FF{k}^{ab}\, \FF{d+1-k}_{ab} dV_{\bar{g}}
\end{align}
is conformally invariant.
\end{proposition}
This proposition follows simply from weight considerations and the
conformal invariance of the conformal fundamental forms. The pairings
above are linked to conformally invariant Dirichlet and Neumann data
for PDE boundary problems constructed around the Fefferman-Graham
obstruction tensor---and similar natural tensors. The main results are
proved in Section~\ref{theorem-proof}. In
Section~\ref{boundary-problems}, we discuss variational problems
linked to these boundary problems and treat the ``grooming'' of the
boundary terms to manifest the pairings~\eqref{pair2}. We then examine
the conditions that lead to well-posed variational problems on
manifolds with boundary. In this same section, we treat in detail the
explicit variation in four dimensions.  In Section~\ref{sec:FG-flat},
we discuss the conditions in which Fefferman--Graham-flat manifolds
with boundary admit Poincar\'e--Einstein structures. In dimension 4,
the Fefferman--Graham tensor is the well-known Bach tensor $B$ and
the Bach-flat condition, namely $B = 0$, is a strict weakening of the
Einstein condition. Thus, in particular, PE manifolds are Bach flat. As mentioned above a PE manifold also necessarily has $\IIo=0$ along the the boundary.  
We prove that this necessary condition is also 
sufficient for a Bach-flat 4-manifold with boundary to admit a
Poincar\'e--Einstein metric on its interior:
\begin{theorem}
Let $(M^4,\cc)$ be a smooth Bach-flat conformal manifold with umbilic boundary. Then the conformal class of metrics on the interior of $M$ contains a metric that is, formally to all orders, Poincar\'e--Einstein.
\end{theorem}
\noindent In fact, in Section~\ref{sec:FG-flat}, we prove a stronger
result that manifests how the metric arises from the conformal class,
see Theorem~\ref{umbilic-PE}.

\subsection{Conventions}
In this article, $M^d$ will refer to a smooth (meaning $C^\infty$) manifold with dimension $d$. When paired with a metric $g$ or an equivalence class $[g] = [\Omega^2 g]$ for $\Omega \in C^\infty M$, such a manifold is a Riemannian or conformal manifold, respectively. Given a Riemannian manifold (or a representative of a conformal manifold), the curvature tensor associated with the Levi-Civita connection $\nabla$ is denoted by $R$ and is defined according to
\begin{align} \label{Riem-curv}
R(x,y)z = (\nabla_x \nabla_y - \nabla_y \nabla_x)z - \nabla_{[x,y]} z\,,
\end{align}
where $x,y,z \in \Gamma(TM)$ are smooth vectors on $M$ and $[\cdot, \cdot]$ is the Lie bracket of vector fields.

For notational convenience, we implement Penrose's abstract index notation, where a subscript Latin index indicates a section of the cotangent bundle, a superscript Latin index indicates a section of the tangent bundle, and multiple indices indicates a section of the corresponding tensor product of the tangent and cotangent bundles, with repeated indices following the Einstein convention. To that end, we may express Equation~(\ref{Riem-curv}) as
$$x^a y^b R_{ab}{}^c{}_d z^d = x^a y^b (\nabla_a \nabla_b - \nabla_b \nabla_a) z^c\,.$$
Furthermore, the metric $g_{ab}$ has an inverse $g^{ab}$ which may be used to raise and lower indices in the usual way. Indices may be symmetrized over (with unit normalization) using round brackets, e.g. $X_{(ab)} = \tfrac{1}{2}(X_{ab} + X_{ba})$, and antisymmetrized with square brackets, e.g. $X_{[ab]} = \tfrac{1}{2}(X_{ab} - X_{ba})$. The tracefree part of a tensor can be indicated via the symbol $\circ$, so that, for example, $X_{(ab)_{\circ}} = \tfrac{1}{2}(X_{ab} + X_{ba} - \tfrac{1}{d}g_{ab} g^{cd} X_{cd})$.

Occasionally, we will use ``dot product'' notation to denote contractions using the appropriate metric; that is, when $v,u \in \Gamma(TM)$, we write $u \cdot v := g(u,v)$. This notation can generalize to a vector contracted into a tensor, so that if $t \in \Gamma(\otimes^k TM)$, then we may write $(v \cdot t)^{a_2 \cdots a_k} := g_{ab} v^a t^{b a_2 \cdots a_k}$.

For conformal geometries, there exists a useful decomposition of the Riemann curvature tensor into the Weyl curvature tensor and the Schouten tensor, given by
$$R_{abcd} = W_{abcd} + g_{ac} P_{bd} - g_{bc} P_{ad} - g_{ad} P_{bc} + g_{bd} P_{ac}\,,$$
where the Schouten tensor is defined according to
$$P_{ab} := \tfrac{1}{d-2} \left(Ric_{ab} - \frac{Sc}{2(d-1)} g_{ab} \right)\,,$$
with the Ricci curvature convention given by $Ric_{ab} = R_{ca}{}^c{}_b$ and scalar curvature $Sc = Ric_a{}^a$. Certain derivatives of the Weyl tensor appear in conformal geometry quite frequently. The Cotton tensor is defined by
$$C_{abc} := \nabla_a P_{bc} - \nabla_b P_{ac} = \tfrac{1}{d-3} \nabla^d W_{dcab}\,,$$
and the Bach tensor is given by
\begin{align} \label{bach-def}
B_{ab} := \Delta P_{ab} - \nabla^c \nabla_a P_{bc} + P^{cd} W_{acbd} = \nabla^c C_{cab} + P^{cd} W_{acbd}\,.
\end{align}
(In the above, $\Delta = g^{ab} \nabla_a \nabla_b$ is the negative Laplacian.)

\subsection{Riemannian hypersurface geometry}
For later use in this article, we record conventions and identities for hypersurface embeddings $\iota : \Sigma \hookrightarrow (M,g)$, which we always assume to be smooth embeddings into compact manifolds. Going forward, we will identity the manifold $\Sigma$ with its image under the embedding $\iota(\Sigma)$.

When $\Sigma$ is the boundary of $M$, we denote by $n \in \Gamma(TM|_{\Sigma})$ the inward pointing unit normal vector. When $\Sigma$ is merely an embedded hypersurface, $n$ denotes some unit normal vector. From this unit normal vector, we may construct the induced metric on $\Sigma$ via
$$\bar{g}_{ab} = \iota^* g = g_{ab}|_{\Sigma} - n_a n_b\,.$$
Then, for any extension of the unit normal vector to $TM$, the second fundamental form may be expressed as
$$\II_{ab} = \nabla_a^\top n_b\,,$$
where $\nabla^\top_a = \bar{g}^b_a \nabla_b$ is the projection of the Levi-Civita connection on $(M,g)$ to $\Sigma$, with the mean curvature defined according to
$$H := \tfrac{1}{d-1} \II_a^a\,.$$
The induced connection on $(\Sigma,\bar{g})$ is related to this projection according to
$$\bar{\nabla}_a v^b \= \nabla_a^\top v^b + n^b \II_{ac} v^c\,,$$
where $v \in \Gamma(T \Sigma)$ and, on the right hand side is extended arbitrarily into $TM$.

The Gauss equation relates the induced curvature to the projection of the curvature, so that
$$\bar{R}_{abcd} = R^\top_{abcd} + \II_{ac} \II_{bd} - \II_{ad} \II_{bc}\,.$$
Following from this equation, we find  the trace-free Fialkow--Gauss equation:
\begin{align} \label{fialkow-gauss}
\IIo^2_{(ab)_{\circ}} - W_{nabn} = (d-3)(P_{ab}^{\otop} - \mathring{\bar{P}} + H \IIo_{ab}) =: (d-3) \Fo_{ab}\,,
\end{align}
where $\Fo_{ab}$ is the trace-free Fialkow tensor. The Codazzi--Mainardi equation, on the other hand, takes the form
$$R_{abcn}^\top = \bar{\nabla}_a \II_{bc} - \bar{\nabla}_b \II_{ac}\,,$$
with its trace-free part taking the form
\begin{align} \label{trfr-Codazzi}
W_{abcn}^\top = \bar{\nabla}_a \IIo_{bc} - \bar{\nabla}_b \IIo_{ac} + \tfrac{1}{d-2} \left(\bar{g}_{bc} \bar{\nabla} \cdot \IIo_a - \bar{g}_{ac} \bar{\nabla} \cdot \IIo_b \right)\,.
\end{align}

\section{Conformal Geometry and Tractor Calculus} \label{sec:tractors}
Here a conformal manifold means a smooth manifold $M$ of
dimension $n\geq 3$ equipped with an equivalence class $\cc$ of Riemanian
metrics, where $g_{ab}$, $\widehat{g}_{ab} \in \cc$ means that
$\widehat{g}_{ab}=\Omega^2 g_{ab}$ for some smooth positive function
$\Omega$.

 On an $n$-manifold $M$ the top exterior power of the tangent bundle $ \Lambda^{n} TM$ is a line bundle. Thus  its square
$\mathcal{K}:=(\Lambda^{n} TM)^{\otimes 2}$  is
canonically oriented and so one can take oriented roots of it. Given
$w\in \mathbb{R}$ we set
\begin{equation} \label{cdensities}
\ce[w]:=\mathcal{K}^{\frac{w}{2n}} ,
\end{equation}
and refer to this as the bundle of conformal densities. For any vector
bundle $\cU$, we write $\cU[w]$ to denote $\cU[w]:=\cU\otimes\ce[w]$.
For example, $\ce_{(ab)}[w]$ denotes the symmetric second tensor power
of the cotangent bundle tensored with $\ce[w]$, i.e. $S^2T^* M\otimes
\ce[w]$ on $M$.

On a fixed Riemannian manifold $\cK$ is canonically
trivialised by the square of the volume form, and so $\cK$ and its
roots are not usually needed explicitly. However if we wish to change
the metric conformally, or work on a conformal structure then these
objects become essential. 
It is straightforward to show a weight $w$ conformal density $\phi$,
that is a section $\phi\in \Gamma(\ce[w])$, can be interpreted as a
double equivalence class $\phi := [g; f] = [\Omega^2 g; \Omega^w f]$
where $f \in C^\infty M$ and we will use this henceforth without further mention.

A weight $w=1$ conformal density $\tau$ will be termed a
\textit{choice of scale} as, at all points where it is non-zero, it
determines a representative of the conformal class of metrics $g_{\tau}$
via the distinguished representative $(g_{\tau} ; 1)$ in the equivalence
class $\tau = [g_{\tau} ; 1]$. When $\tau > 0$ we refer to it as a
\textit{true scale}.
On a conformal manifold there is evidently a canonical section $\bg$
of $\odot^2 T^* M[2]$ given by $\bg = [g; g] $. This is called the
conformal metric.
The metric determined by $\tau\in \Gamma(\ce[1])$ then arises as  $g_{\tau} =
\tau^{-2}\bg$. In this article, when we do explicit computations we will pick a metric and trivialize density bundles accordingly.

Given a true scale $\tau$ one has the corresponding Levi-Civita
connection $\nabla^{g_{\tau}}$. Clearly this acts naturally on
densities and, in fact, as an operator on $\Gamma(\ce[w])$, is
recovered by the formula $\nabla^{\tau} = \tau^w \circ d \circ
\tau^{-w}$. It follows that $\nabla^{g_{\tau}}$ preserves the
conformal metric $\bg$ and so we may use that to raise and lower
indices without confusion. It turns out that this simplifies many
computations, and so we will do this.

On a general conformal manifold $(M,\cc)$, there is no
distinguished connection on $TM$. But there is an invariant and
canonical connection on a closely related bundle, namely the conformal
tractor connection on the standard tractor bundle.
Here we review the basic conformal tractor calculus on
Riemannian and conformal manifolds. See
\cite{BEG,curry-G,GP} for more details.
Unless stated otherwise, calculations will be done with the use of generic $g \in
\cc$.

Given a metric
representative $g \in \cc$, the \textit{standard tractor bundle} $\mathcal{T}M$
is isomorphic to the rank $d+2$ vector bundle
$$\mathcal{T}M \stackrel{g}{\cong} \ce M[1] \oplus TM[-1] \oplus \ce M[-1]\,.$$
Sections of this bundle (and tensor products involving this bundle) are denoted using capital latin letters for abstract indices, so that for example
\begin{equation}\label{Tsplit}
  T^A \stackrel{g}{=} (\tau^+, \tau^a, \tau^-)\,.
  \end{equation}
  While the
isomorphism holds for any metric representative, the isomorphism
transforms non-trivially under conformal rescalings:
\begin{equation}\label{ct}
  T^A \stackrel{\Omega^2 g}{=} (\tau^+, \tau^a + \Upsilon^a \tau^+, \tau^- - \tau^a \Upsilon_a - \tfrac{1}{2} |\Upsilon|^2_g \tau^+)\,,
  \end{equation}
where $\Upsilon := d \log \Omega$.

The tractor bundle decomposition determined by a metric $g\in\cc$, as given in \eqref{Tsplit}, is captured by three bundle maps or
\textit{injectors}: $X^A \in \Gamma(\ct
M[1])$, $Z^A_a \in \Gamma(T^* M \otimes \ct M[1])$, and $Y^A \in
\Gamma(\ct M[-1])$ characterised by
$$T^A \stackrel{g}{=} \tau^+ Y^A + \tau^a Z^A_a + \tau^- X^A\,.$$
These have conformal transformations determined by \eqref{ct}. In particular $X^A$ is  independent of the
 choice of metric representative; we
label it the \textit{canonical tractor}. The other injectors depend on
a the choice of metric representative. Nevertheless another choice-independent
tractor may also be constructed from these injectors: the tractor
metric, expressed in a choice of metric representative $g \in \cc$:
$$h^{AB} \stackrel{g}{:=} \bg^{ab} Z^A_a Z^B_b + 2 X^{(A} Y^{B)} \in \Gamma(\odot^2 \ct M)\,.$$
As it is invertible this identifies $\ct M$ with its dual $\ct^*M$.

Additionally, one may use these injectors to construct sections of tractor bundles from sections of tensor products of the (co)tangent bundle (or the weighted equivalent). Doing so invokes differential operators known as \textit{splitting operators} (see~\cite{josef} for a complete characterization). It is sufficient for us to note that these are operators with left-inverses. One relevant example of a splitting operator takes the form
\begin{equation} \label{q-definition}
\begin{aligned}
q_{AB}^{ab} : \Gamma(\odot^2_{\circ} T^* M[w+2]) &\rightarrow \Gamma(\odot^2_{\circ} \mathcal{T} M[w]) \\
t_{ab} &\mapsto Z_A^a Z_B^b t_{ab} - \tfrac{2}{d+w} X_{(A} Z_{B)}^a \nabla^b t_{ab} \\
&\phantom{\mapsto} + \tfrac{1}{(d+w)(d+w-1)} X_A X_B (\nabla^a \nabla^b t_{ab} + (d+w) P^{ab} t_{ab})\,.
\end{aligned}
\end{equation}
Note that this operator lives in the kernel of contraction with $X$.

The
\textit{tractor connection} $\nabla^\ct$ is defined, in a choice of scale
$\tau$, by
$$\nabla_a^\ct T^B \stackrel{g}{=} (\nabla_a^{\tau} \tau^+ - \bg_{ac}
\tau^c, \nabla_a^{\tau} \tau^b + \bg_a^b \tau^- + (P^{g_{\tau}})_a^b
\tau^+, \nabla_a^{\tau} \tau^- - (P^{g_{\tau}})_{ab} \tau^b)\,.$$ It
is straightforward to check via the tractor transformation identity
that $\nabla^\ct$ is indeed a choice-independent
connection on conformal manifolds.

From its definition, one observes that, for
a choice of metric $g \in \cc$,
\begin{align*}
\nabla_b^\ct X^A =& Z^A_b \\
\nabla_b Z^A_a =& -(P^g)_{ab} X^A - \bg_{ab} Y^A \\
\nabla_b Y^A =& (P^g)_b^a Z^A_a\,.
\end{align*}
where on the left hand side we have coupled the tractor connection to
the Levi-connection of $g_\tau$. These relations are useful for
computations and, in particular, using these one can verify that the
tractor connection preserves the tractor metric
$$
\nabla^\tau_a h_{BC}=0.
$$
It follows that it sensible to use $h_{BC}$ and its inverse to lower and raise indices.

A non-trivial fact is that one may couple the tractor connection and
the Levi-Civita connection in a way to produce a tractor-valued
differential operator, called the \textit{Thomas-$D$ operator}, that
acts on any weighted tractor bundle, and is independent of the choice
of metric representative. This is defined, in a choice of scale, by
$$D^A T \stackrel{g_{\tau}}{:=} (d+2w-2) wY^A T + (d+2w-2) \bg^{ab} Z^A_a \nabla_b^{\ct} T - X^A (\Delta^\ct + w J^{g_{\tau}}) T\, ,$$
where $T$ is weight-$w$ section of a tractor bundle.

Closely linked to the tractor connection (see e.g.\ \cite{BEG}) is the
conformally invariant differential operator $D_{ab}:\Gamma(\ce[1])\to
\odot_o^2T^*M[1])$ given by
\begin{equation}\label{Eop}
D_{ab}\tau=\nabla_a\nabla_b \tau +\mathring{P}_{ab} \tau .
\end{equation}
This is sometimes called the {\em almost-Einstein operator} because non-zero
solutions to $D_{ab}\tau=0$ are scales such that $g=\tau^{-2}\bg$ is
Einstein on the open dense set (we assume $M$ connected) where $\tau$ is non-zero, and such solutions correspond to parallel standard tractors \cite{GoverNurowski,GoverAE}.

\medskip

Next, for our purposes it is useful to recall from \cite{BEG} the
notion of the normal tractor. This is a natural section of $\ct|_{\partial M}$ along the
boundary $\partial M$ of a conformal manifold $M$ with boundary. In a choice of
metric representative, this is given by
$$N^A \stackrel{g}{:=} (0, n^a, -H^g) \in \Gamma(\ct M|_{\Sigma})\,,$$
where $n$ is the inward-pointing unit normal vector to the boundary
$\Sigma$ satisfying $|n|^2_g = 1$ and $H^g$ is its mean curvature for
the representative $g$. Elsewhere, we use the same symbol $n$ to refer to any extension of $n$ off $\Sigma$ unless specified. It is easily checked that the display
transforms according to \eqref{ct}, and so $N^A$ is a conformally
invariant tractor field. 

Now, given a conformal hypersurface embedding $\Sigma \hookrightarrow (M^d,\cc)$, one may always (formally) solve the singular Yamabe problem to an order determined by the dimension of the manifold, see \cite{ACF}, and, for the way we describe it here, \cite[Theorem 4.5]{will1}:
\begin{proposition} \label{singyam}
Let $\Sigma \hookrightarrow (M^d,\cc)$ be a smooth conformal hypersurface embedding with $\sigma_0 := [g;s] \in \Gamma(\ce M[1])$, where $s$ is any defining function for $\Sigma$. Then there exists a distinguished density $\sigma \in \Gamma(\ce M[1])$ unique modulo terms of order $\mathcal{O}(\sigma^{d+1}_0)$ satisfying
$$\tfrac{1}{d^2} |D \sigma|^2 = 1 + \sigma^d \mathcal{B}_\sigma\,,$$
for some smooth density $\mathcal{B}_{\sigma} \in \Gamma(\ce M[-d])$.
\end{proposition}
We call the density $\sigma$, distinguished by the above proposition,
the \textit{singular Yamabe density}, the metric it determines
$g^o:=g/s^2$ the \textit{singular Yamabe metric}, and
$\mathcal{B}_{\sigma}|_{\Sigma} \in \Gamma(\ce \Sigma[-d])$ the
\textit{obstruction density}; note that the obstruction density is
independent of the choice of $\sigma_0$ and thus is an invariant of
the embedding.

From this result, and the conformal invariance of the almost-Einstein operator \eqref{Eop}, it follows immediately 
that associated to any conformal hypersurface embedding one may
construct a canonical \textit{almost-Einstein tensor} $E_{ab}$, unique to order
$\sigma^{d-1}$:
\begin{proposition}\label{Edef}
Let $\Sigma \hookrightarrow (M^d,\cc)$ be a smooth conformal
hypersurface embedding with singular Yamabe density $\sigma := [g; s]
\in \Gamma(\ce M[1])$. Then
$$E_{ab} := D_{ab}\sigma $$
is uniquely determined modulo terms of order $\mathcal{O}(\sigma^{d-1})$.
\end{proposition}

The above results are  local. But from the works of~\cite{Loewner,aviles1985,ACF} (see the discussion in the introduction of~\cite{Graham2016} for more details), there is a stronger, global result:
\begin{proposition} \label{unique-singyam}
On a compact conformal manifold with boundary, there exists a unique
solution $\sigma$ to the singular Yamabe problem. Furthermore, if the
obstruction density vanishes, then that solution is smooth.
\end{proposition}
\begin{remark}
For convenience, we will refer to a singular Yamabe metric as smooth when the singular Yamabe density it arises from is smooth.
\end{remark}

As a corollary, when the obstruction density vanishes, the almost Einstein tensor $E$ above is determined uniquely by the conformal embedding.
\begin{corollary} \label{unique-E}
Let $(M^d,\cc)$ be a compact conformal manifold with smooth boundary embedding $\Sigma \hookrightarrow (M^d,\cc)$ and vanishing obstruction density. Then, $E \in \Gamma(\odot^2_{\circ} T^* M[1])$ is smooth up to the boundary and is uniquely determined by the boundary embedding and the conformal structure $(M,\cc)$.
\end{corollary}
\begin{proof}
From Proposition~\ref{unique-singyam}, we have that there exists a
unique smooth singular Yamabe density $\sigma$. But $E$ is defined by
a conformally invariant differential operator on $\sigma$, and hence
as $\sigma$ only depends on the boundary embedding and the conformal
structure, so does $E$. Furthermore, as $\sigma$ is smooth up to the
boundary (and the underlying conformal structure is also), so is $E$.


\end{proof}

\section{Fefferman--Graham Obstruction Tensors and the Main Results} \label{theorem-proof}

In~\cite{FG,FGbook}, it is shown that on each even-dimensional conformal manifold $(M^d,\cc)$, the so-called \textit{Fefferman--Graham} obstruction tensor $\mathcal{O}^{(d)}$ arises as a distinguished, divergence-free section of $\odot^2_\circ T^* M[2-d]$. These tensors get their name from the fact that they appear as the smoothness obstruction to a formal solution of the Poincar\'e--Einstein equation on a $(d+1)$ dimensional manifold. More specifically, given a conformal manifold embedded as the boundary of a smooth manifold  $(\Sigma,\bar{\cc}) \hookrightarrow M^d$ (with $d$ odd), the existence of a formal extension $\cc$ of $\bar{\cc}$ to $M$ containing a Poincar\'e--Einstein metric as a reprensentative is obstructed by the Fefferman--Graham obstruction tensor of $(\Sigma,\bar{\cc})$. These tensors are of special interest as, aside from the Weyl tensor, they are the only conformally-invariant curvature tensors that are linear in the curvature at leading order~\cite{FG}. For our purposes, this fact is critical.

The Fefferman--Graham obstruction tensor also arises variationally from a conformally invariant energy functional
\begin{align} \label{Q-integral}
\int_{M^d} Q_d\, dV
\end{align}
where $Q_d$ is Branson's $Q$-curvature in $d$ dimensions~\cite{GrahamHirachi}. While $Q$-curvatures are not conformally-invariant themselves, they transform invariantly with weight $-d$ modulo a divergence---hence, on a closed manifold, Display~(\ref{Q-integral}) is an integrated conformal invariant. However, for our purposes, we require an integral whose integrand is pointwise conformally-invariant and whose variation is the Fefferman--Graham obstruction tensor. That such a pointwise-invariant exists is captured in the following two propositions.

\begin{proposition} \label{d-invariant}
Let $(M^d,\cc)$ be an even-dimensional conformal manifold. Then, the ambient scalar invariant $L_d := |\tilde{\nabla}^{\frac{d-4}{2}} \tilde{R}|^2 \in \Gamma(\ce M[-d])$ is the unique (modulo terms cubic or higher order in curvature) natural conformal invariant that is quadratic in curvature at this weight.
\end{proposition}
\begin{proof}
From the construction provided by Fefferman and Graham~\cite[Proposition 9.1]{FGbook}, the ambient invariant $L_d := |\tilde{\nabla}^{\frac{d-4}{2}} \tilde{R}|^2$ is a scalar conformal invariant. Furthermore, it is straightforward to check that it is quadratic at leading order.
It thus remains to show that this conformal invariant is unique.

To do so, we note from the construction provided by Fefferman and Graham, that, for a fixed $d$, there exists a dimension-parametrized family of formulas for conformal invariants $\{L_d^{d'}\}_{d'}$ where for each even $d' > d$, the formula is given by
$$L_d^{d'} := |\tilde{\nabla}^{\frac{d-4}{2}} \tilde{R} |^2\,.$$
Now, because the ambient curvature has vanishing Laplacian and divergence in sufficiently high dimension and a vanishing Ricci curvature to sufficiently high order, it follows that the ambient tensor $\tilde{\nabla}^{\frac{d-4}{2}} \tilde{R}$ is trace-free modulo lower-order curvature terms. Furthermore, from the Bianchi identity and the Ricci identity, it follows that we may not antisymmetrize over more than two indices without this expression being expressible in terms of lower-order invariants. It follows from straightforward representation theory that $L_d^{d'}$ is a formula for a unique (modulo terms cubic or higher order in the curvature) scalar conformal invariant for all $d' > d$.

Now for contradiction suppose that there are two distinct invariants (at quadratic order in curvature) of the same form when $d' =d$. Then, those two invariants would be realized as formulas that produce two distinct families of conformal invariants all even dimensions $d' > d$. However, from the uniquness argument above in $d' > d$, these families of invariants must agree at quadratic order in curvature. This then implies that there exist at least one family of curvature identities that hold in all dimensions $d' > d$ that does not hold in $d$ dimensions. But no such curvature identities exist---a contradiction. This establishes that the invariant $L_d \in \Gamma(\ce M[-d])$ is the unique conformal invariant (of this weight that is quadratic at leading order) constructible from the ambient metric. However, from~\cite[Theorem 9.4]{FGbook}, all natural conformal invariants at this weight are constructible from the ambient metric. So uniqueness follows.
\end{proof}

Low-lying invariants of this form are known explicitly. In four dimensions, we have the (clearly) invariant density $L_4 := |W|^2 \in \Gamma(\ce M[-4])$. In six dimensions, Fefferman and Graham presented another invariant,
$$L_6 := |V|^2 + 16 W^{abcd} (\nabla_a C_{dcb} - P_a^e W_{ebcd}) + 16 C_{abc} C^{abc}\,,$$
where $V_{abcde}$ has leading term $\nabla_a W_{bcde}$ and subleading terms that are pure trace. 

Now we show that it is precisely this family of pointwise conformal invariants whose integrals vary into the Fefferman-Graham obstruction tensors.
\begin{proposition}\label{obstruction-variation}
Let $(M^d,g)$ with $d$ even be a manifold without boundary. Then, there exists at least one pointwise conformal invariant $\mathcal{L}_d \in \Gamma(\ce M[-d])$ that is quadratic in curvature at leading order such that
$$\delta \int_M \mathcal{L}_d dV = \int_M (\delta g)^{ab} \mathcal{O}^{(d)}_{ab} dV_g\,,$$
where $\mathcal{O}^{(d)}$ is the Fefferman--Graham obstruction tensor in $d$ dimensions. Furthermore, all such pointwise conformal invariants are proportional to $L_d$ modulo terms of cubic or higher order in curvature, and their functional $L^2$-gradient are the Fefferman--Graham obstruction tensor modulo terms quadratic in curvature.
\end{proposition}
\begin{proof}
From~\cite{GrahamHirachi}, there exists a natural scalar invariant called the Branson $Q$-curvature sastisfying
$$\delta \int_M Q^g dV_g = \int_M h^{ab} \mathcal{O}^{(d)}_{ab} dV_g\,.$$
Furthermore, $\int_M Q^g dV_g$ is a global conformal invariant. Now from~\cite{Alexakis} it thus follows that the $Q$-curvature is decomposable as a sum of a Pfaffian, a pointwise conformal invariant, and a divergence. As the Pfaffian term is topological, it cannot contribute to the metric variation. Furthermore, divergences can be omitted because the manifold is closed.
Thus, there exists a pointwise conformal invariant that varies into the Fefferman--Graham obstruction tensor. By fiat, we identify an arbitrary such pointwise conformal invariant with $\mathcal{L}_d$.

Now, at leading order it is known that $\mathcal{O}^{(d)} \propto \Delta^{\frac{d-4}{2}} B$.
But because $\mathcal{O}^{(d)}$ is linear at leading order in curvature, it can only arise variationally from curvature invariants that are at most quadratic (at leading order) in curvature. However, because the only conformal invariants that are linear in curvature are the Weyl curvature and the obstruction tensor~\cite{FG}, there are no conformally-invariant scalars that are linear in curvature. Hence, it follows that $\mathcal{O}^{(d)}$ arises variationally from a conformal invariant that is quadratic in curvature at leading order. Thus $\mathcal{L}_d$ is as required by the proposition. Furthermore, by the uniqueness provided in Proposition~\ref{d-invariant}, all possible choices of $\mathcal{L}_d$ must be proportional to $L_d$ modulo terms that are at least cubic in curvature.
\end{proof}

That such pointwise conformal-invariants exist will be instrumental in constructing $\FF{d}$. However, before we may proceed with the proof of Theorem~\ref{main}, we first need to provide some technical results regarding the existence of certain conformally-invariant normal derivative operators.

To construct the required normal derivative operators, we reproduce the following propositions.
\begin{proposition}[Proposition 5.8 of~\cite{GPt}] \label{lower-normal-derivs}
Let $k \in \mathbb{Z}_+$ be given. Then there exists a family of conformally invariant operators $\delta_K : \Gamma(\ct^\Phi M[w]) \rightarrow \Gamma(\ct^{\Phi}[w-K])|_{\Sigma}$ defined by
$$\delta_K := N^{A_2} \cdots N^{A_k} \delta_R D_{A_2} \cdots D_{A_k}\,,$$
with transverse order $K$ so long as
$$w \not \in \left\{ \tfrac{2K-1-d}{2}, \tfrac{2K-2-d}{2},\cdots, \tfrac{K+1-d}{2} \right\}\,.$$
Here, $\delta_R : \Gamma(\ct^\Phi M[w]) \rightarrow \Gamma(\ct^{\Phi} M[w-1])|_{\Sigma}$ and is defined by $\delta_R := \nabla_{n} - w H$.
\end{proposition}
\begin{proposition}[Theorem 3.4 of~\cite{Blitz1}] \label{higher-derivs-ops}
Let $J,k \in \mathbb{Z}_+$ such that $0<J$, $0<k<d/2$ and let $d$ be even. Then, there exists a family of conformally invariant differential operators $\delta_{J,k} : \Gamma(\ct^\Phi M[w]) \rightarrow \Gamma(\ct^\Phi M[w-k-J])|_{\Sigma}$ determined as follows. For $k \leq J$,
$$\delta_{J,k} = N^{A_1} \cdots N^{A_k} \delta_J P_{A_1 \cdots A_k}\,,$$
where $P_{A_1 \cdots A_k}$ is a tractor-valued operator with leading derivative term equal to $D_{A_1} \cdots D_{A_k}$, and all subleading terms depend at least linearly on the $W$-tractor.
For $k > J$, then $\delta_{J,k}$ is determined by the equation
$$(d+2w-2k) \delta_{J,k} = N^{A_1} \cdots N^{A_k} \delta_J P_{A_1 \cdots A_k}\,.$$
When $w = k-d/2$, $\delta_{J,k}$ has transverse order $J+k$.
\end{proposition}

Notably, these results are concerned with conformally-invariant normal derivative operators on tensor products of (weighted) tractor bundles. For our purposes, we require the existence of the same on tensor products of (weighted) \textit{(co)tangent} bundles---in particular, we require transverse derivatives on conformal metric variations $\delta g := h \in \Gamma(\odot^2 T^* M[2])$. That the required operators exist up to order $d-3$ is proved in the following proposition.

\begin{proposition} \label{normal-operators}
Let $\Sigma \hookrightarrow (M^d,\cc)$ be a conformally-embedded hypersurface with $d$ even. Then, for each $2 \leq k \leq d-3$, there exists a family of operators
\begin{align*}
\delta_{\otop}^{(k)} :& \Gamma(\odot^2_\circ T^*M[2]) \rightarrow \Gamma(\odot^2_\circ T^*\Sigma[2-k]) \\
\delta_{n}^{(k)} :& \Gamma(\odot^2_\circ T^*M[2]) \rightarrow \Gamma(T^*\Sigma[1-k]) \\
\delta_{nn}^{(k)} :& \Gamma(\odot^2_\circ T^*M[2]) \rightarrow \Gamma(\ce \Sigma[-k])\,,
\end{align*}
with leading transverse terms $\otop \circ \nabla_n^k$, $\top \circ (n \cdot \nabla_n^k)$, and $(n \otimes n) \cdot \nabla_n^k$, respectively.

Furthermore, there exist conformally-invariant operators
\begin{align*}
\delta_{\otop}^{(1)} :& \Gamma(\odot^2_\circ T^*M[2]) \rightarrow \Gamma(\odot^2_\circ T^*\Sigma[1]) \\
\delta_{nn}^{(1)} :& \Gamma(\odot^2_\circ T^*M[2]) \rightarrow \Gamma(\ce \Sigma[-1])\,,
\end{align*}
with leading transverse terms $\otop \nabla_n$ and $(n \otimes n) \cdot \nabla_n$ respectively.
\end{proposition}
\begin{proof}
To prove the proposition, we construct the operators acting on $h \in \Gamma(\odot^2_\circ T^*M[2])$ according to the following formulae and then verify that they have the correct leading transverse order term:
\begin{align*}
\delta^{(k)}_{\otop} h :=& -\tfrac{2}{d-2} (\bar{Q}^*_2 \circ \otop \delta^{k} D_{[A} q^{bc}_{B]C})(h_{bc}) \\
\delta^{(k)}_n h :=& -\tfrac{2}{d-2} (\bar{Q}^*_1 \circ \top N^B \delta^k  D_{[A} q^{bc}_{B]C})(h_{bc}) \\
\delta^{(k)}_{nn} h :=& -\tfrac{2}{d-2} (X^A N^B N^C \delta^k D_{[A} q^{bc}_{B]C})(h_{bc})\,,
\end{align*}
where $1 \leq k \leq d-3$ for the first and third operators,$2 \leq k \leq d-3$ for the second operator, and $q^{ab}_{AB}$ is defined as in Equation~(\ref{q-definition}). In the above, for $1 \leq k \leq \frac{d}{2} -1$, we define $\delta^k := \delta_k$ and for $\tfrac{d}{2} \leq k \leq d-3$, we define $\delta^k := \delta_{k-d/2 + 1,d/2-1}$, both of which are transverse-order $k$ normal derivative operators. Further, $\bar{Q}^*_1$ and $\bar{Q}^*_2$ are defined below (by construction). For brevity in this proof, we define $H_{AB} := q^{ab}_{AB}(h_{ab})$ and $\mathcal{H}_{ABC} := -\tfrac{2}{d-2}D_{[A} H_{B]C}$. 

Using the formula for $q_{AB}^{ab}$ in Section~\ref{sec:tractors}, $H$ is well-defined when $d \neq 0,1$. Next, a short calculation yields 
\begin{align*}
X^A Z^B_b Z^C_c \mathcal{H}_{ABC} &= -\tfrac{1}{d-2} X^A Z^B_b Z^C_c D_B H_{AC} \\
&= - X^A Z^C_c \nabla_b H_{AC} \\
&= h_{bc}\,.
\end{align*}
Furthermore, it is easy to check that $X^A Z^B_b X^C \mathcal{H}_{ABC} = 0$. It follows that the conformally-invariant part of $\mathcal{H}_{ABC}$ is proportional to $Y_{[A} Z_{B]} Z_{C}$.

We must now show that $\delta^k$ maps the $Y \wedge Z \otimes Z$ slot to $k$ normal derivatives of $h$ to leading order in the same slot. First, by the definition given, when $k \leq \tfrac{d}{2} -1$, it follows from Proposition~\ref{lower-normal-derivs} that $\delta^k = \nabla_n^k + \text{lower order terms}$. Similarly, from Proposition~\ref{higher-derivs-ops}, when $\tfrac{d}{2} \leq k \leq d-3$, $\delta^k = \nabla_n^k + \text{lower order terms}$. This is necessary to show that $\delta^k \mathcal{H}$ may contain a term of the form $Y_{[A} Z_{B]}^a Z_{C}^b \nabla_n^k h_{ab}$. To check that $\delta^k \mathcal{H}$ indeed does contain such a term, it suffices to verify that other terms do not cause a cancellation of its coefficient.

Now, by explicitly computing commutators, we have the following identities for $k \geq 3$,
\begin{align*}
[\nabla_n^k, X_A] &= k Z_{An} \nabla_n^{k-1} - \tfrac{1}{2} k(k-1) Y_A \nabla_n^{k-2} + \mathcal{O}(\nabla_n^{k-3}) \\
[\nabla_n^k, Y_A] &= \mathcal{O}(\nabla_n^{k-2}) \\
[\nabla_n^k, Z_A^a] &= -k n^a Y_A \nabla_n^{k-1} + \mathcal{O}(\nabla_n^{k-2})\,,
\end{align*}
Note these identities also hold for $k = 2$, except
$$[\nabla_n^2, X_A] = 2 Z_{An} \nabla_n -  Y_A \,.$$
In the above, we have dropped any terms that involve curvatures or derivatives of $n$  (as by weight considerations they cannot be of sufficiently high transverse order to contribute to the leading order structures in the desired operators). Using these, we have that for $k \geq 2$,
\begin{equation} \label{XZZ-H}
\begin{aligned}
X^A Z^B_b Z^C_c \nabla_n^k \mathcal{H}_{ABC} =& -\tfrac{(k-1)(2d-k-4)}{2(d-2)^2} \nabla_n^k h_{bc} 
+ \tfrac{k(k-1)(2d-k-4)}{2d(d-2)^2} n_c n^a \nabla_n^k h_{ba} 
\\&+ \tfrac{k(2d^2 - dk - 5d+4)}{2d(d-2)^2} n_b n^a \nabla_n^k h_{ca}
 - \tfrac{k(2d^2 -dk - 5d+4)}{2d(d-1)(d-2)^2} g_{ab} n^c n^d \nabla_n^k h_{cd} \\&- \tfrac{k(k-1)(2d-k-4)}{2d(d-1)(d-2)} n_a n_b n^c n^d \nabla_n^k h_{cd} 
\\&+ \text{lower order terms.}
\end{aligned}
\end{equation}
The first, second, and fifth terms in this expression only vanish when $k=2d-4$, which for $k \leq d-3$ and $d \geq 4$ is never satisfied. Similarly, the third and fourth terms only vanish for $k \geq 2$ when $k=4$ and $d=4$, which is disallowed because $k \leq d-3$. We have thus established that $\delta^k \mathcal{H}$ contains a term proportional to a non-zero constant times $\nabla_n^k h$. It remains to check that the remainder of the compositions necessary to construct the operators $\delta^{(k)}$, $\delta^{(k)}_n$, and $\delta^{(k)}_{nn}$ do not ruin this feature.

Now, because we are interested in irreducible projections to the boundary of these normal derivative operators, we can contract with the normal tractor $N$ and take the relevant traces. So, again keeping only leading normal derivative terms, we find that
$$\otop X^A Z^B_b Z^C_c \nabla_n^k \mathcal{H}_{ABC} \= -\tfrac{(k-1)(2d-k-4)}{2(d-2)^2} \otop \nabla_n^k h_{bc}\,.$$
Similarly, by multiplying Equation~(\ref{XZZ-H}) by $n^b$ and $n^b n^c$ respectively, we obtain
$$\top X^A N^B Z^C_c \nabla_n^k \mathcal{H}_{ABC} \= \tfrac{d-k}{d(d-2)} n^a \bar{g}^b_c \nabla_n^k h_{ab}\,,$$
and
$$ X^A N^B N^C \nabla_n^k \mathcal{H}_{ABC} \= \tfrac{(d-k)(2d^2 -2dk + k^2 -6d+3k+4)}{2d(d-1)(d-2)^2} n^b n^c \nabla_n^k h_{bc}\,.$$
Note that all of these have nonzero leading terms for the allowed values of $k$.

The final step is to show that we may extract these leading terms from the tractor construction via operators labelled $\bar{Q}^*$. Note that this is only required for the two operators taking values in tensor-valued densities---for the operators that map $h$ to a scalar, we are already done. Now, for those operators $\delta^{(k)}_{\otop}$ and $\delta^{(k)}_n$, we extract the appropriate normal derivative operators by contracting the tractor with $X$ and then applying the formal adjoint $\bar{q}^*$ of the relevant insertion operator $\bar{q}$. On a symmetric trace-free rank-2 tractor $W^{AB} \in \Gamma(\odot^2_{\circ} \ct \Sigma[w])$, this formal adjoint operator takes the form~\cite{josef}
\begin{multline*}
\bar{q}_2^*(W^{AB}) = \Big(w(w+1) \bar{Z}_A^{(a} \bar{Z}_B^{b)_{\circ}} - 2(w+1) \bar{\nabla}^{(a} \bar{Z}_A^{b)_{\circ}} X_B \\+ \bar{\nabla}^{(a} \bar{\nabla}^{b)_{\circ}} X_A X_B - w \bar{P}^{(ab)_{\circ}} X_A X_B \Big) W^{AB}\,.
\end{multline*}
Similarly, for a tractor vector $V^A \in \Gamma(\ct \Sigma[w])$,
\begin{align*}
\bar{q}_1^*(V^A) = \left(-(w+1) \bar{Z}_A^a + \bar{\nabla}^a X_A \right) V^A\,.
\end{align*}
Observe that when $X_A X_B W^{AB} = 0$, the extraction operator $\bar{q}^*_2$ may further be defined when $w = -1$ by weight continuation. 

Now we define
\begin{align*}
\bar{Q}^*_2 (T_{ABC}) &:= \bar{q}_2^* \circ \otop ( X^A T_{A(BC)_{\circ}}) \\
\bar{Q}^*_1 (T_{ABC} ) &:= \bar{q}^*_1 \circ \top (X^B N^C T_{ABC})
\end{align*}
for tractor sections of $\otimes^3 \mathcal{T} M[-k-1]$. Observe that for values of $k \geq 2$, the above operators containing $\bar{q}^*_1$ and $\bar{q}^*_2$ are well-defined and extract the desired tensor. It follows that $\delta^{(k)}_{\otop}$ and $\delta^{(k)}_n$ exist. On the other hand, for $k = 1$, observe that $X^A X^B \delta^1 \mathcal{H}_{ABC} = 0$ from the symmetry of $\mathcal{H}$, and so $\bar{Q}^*_2 (\delta^1 \mathcal{H}_{ABC})$ is well-defined as well. This ensures that $\delta^{(1)}_{\otop}$ exists. This completes the proof.

\end{proof}

We are now ready to prove the main result.

\begin{proof}[Proof of Theorem~\ref{main}]
As noted in Proposition~\ref{obstruction-variation}, on a manifold without boundary, $\int_M \mathcal{L}_d dV$ varies into the obstruction tensor. Note that because $\mathcal{L}_d$ is quadratic in curvature at leading order, integration by parts occurs at least twice in order to obtain the leading structure of the obstruction $\int h^{ab} (\Delta^{\frac{d-4}{2}} B_{ab} + \text{ subleading})$. Now such a term can only arise (up to identities) at leading order via integration by parts from an integral of the form
$$-\int \left( \alpha \nabla_c \Delta^{\frac{d-6}{2}} B_{ab} + (1-\alpha) \Delta^{\frac{d-4}{2}} C_{cab} \right) \nabla^c h^{ab}\,,$$
for some constant $\alpha$, which follows because $B_{ab} = \nabla^c C_{cab} + \text{lower order}$.
Indeed, after integrating by parts, we are left with (modulo subleading terms),
$$\int \left[ h^{ab} \Delta^{\frac{d-4}{2}} B_{ab} -\nabla^c \left( \alpha h^{ab} \nabla_c \Delta^{\frac{d-6}{2}} B_{ab}+ (1-\alpha) h^{ab} \Delta^{\frac{d-4}{2}} C_{cab}  \right) \right]\,,$$
as required.

Now for a manifold without boundary, the divergence term vanishes---however, when the manifold has a boundary, the boundary term (again, modulo subleading terms) takes the form
\begin{align*}
&\int_{\Sigma} n^c h^{ab}\left( \alpha  \nabla_c \Delta^{\frac{d-6}{2}} B_{ab}+ (1-\alpha) \Delta^{\frac{d-4}{2}} C_{cab}  \right) \\=&  \int_{\Sigma} n^c h^{ab}\left( \alpha  \nabla_c \Delta^{\frac{d-6}{2}} \nabla^d C_{dab}+ (1-\alpha) \Delta^{\frac{d-4}{2}} C_{cab}  \right)   + \text{subleading} \\
=&  \int_{\Sigma} \alpha n^d h^{ab}  \nabla_n^2 \Delta^{\frac{d-6}{2}} C_{dab}+ (1-\alpha) n^c h^{ab} \Delta^{\frac{d-4}{2}} C_{cab} + \text{subleading} \\
=& \int_{\Sigma} h^{ab} n^d \Delta^{\frac{d-4}{2}} C_{dab}+ \text{subleading}\,,
\end{align*}
where $n$ is an inward pointing unit normal vector.
In particular, this term has a non-vanishing coefficient. Our goal now is to show that this leading transverse-order integrand on the boundary combines with lower order terms to produce a pointwise conformal invariant with leading structure $n^d \Delta^{\frac{d-4}{2}} C_{d(ab)}$ contracted into the metric variation along the boundary.
%

By construction, the variation of the integral in question is conformally invariant as is the bulk integral itself (which can be established by examining the variation of the integral on a manifold without boundary). It follows that the boundary integral arising from the variation is separately conformally invariant.


To show pointwise invariance of the integrand along the boundary, we now must decompose the integral along $\Sigma$.

In general, the integrand of the boundary integral can be expressed in terms of bulk curvatures, the normal vector, the metric variation, and their derivatives. Furthermore, those derivatives can be decomposed into intrinsic hypersurface derivatives and transverse derivatives. Moreover because the boundary has no boundary, we may integrate by parts to remove intrinsic hypersurface derivatives from the metric variation, leaving only transverse derivatives on metric variations contracted into curvatures (whether intrinsic or extrinsic) and copies of the normal vector. Then we perform a tangential decomposition of those transverse derivatives on metric variations into four independent components:
$$\otop \nabla_n^k \mathring{h}_{ab} \qquad \top n^b \nabla_n^k \mathring{h}_{ab} \qquad n^a n^b \nabla_n^k \mathring{h}_{ab} \qquad \nabla_n^k h_a^a\,.$$

Observe that the first integration by parts to remove derivatives from $h$ may leave up to $d-3$ derivatives on $h$, as schematically terms of the form $(\nabla^{m} W)~\cdot~(\nabla^{\ell+2} h)$ with $m + \ell = d-4$ can appear in $\mathcal{L}_d$.  So we can then the boundary integral as
$$\int_{\Sigma} dV_{\bar g} \sum_{i=0}^{d-3} \left(\bar{O}_{(i)}\nabla_n^i h^a_a + \bar{P}_{(i)} n^a n^b \nabla_n^i \mathring{h}_{ab} + \bar{Q}_{(i)}^a \top (n^b \nabla_n^i \mathring{h}_{ab} )+ \bar{R}_{(i)}^{ab} \otop \nabla_n^i \mathring{h}_{ab} \right)$$
where $\bar{O}_{(i)},\bar{P}_{(i)} \in C^\infty \Sigma$, $\bar{Q}^a_{(i)} \in \Gamma(T\Sigma)$, and $\bar{R}^{ab}_{(i)} \in \Gamma(\odot^2_{\circ} T\Sigma)$ are composed of intrinsic and extrinsic curvatures and their derivatives for each $i \in \{0,\ldots, d-3\}$. 

To simplify this integral, we may first consider just those variations $h = fg$, i.e. those that are pure trace. However, note that such a variation amounts to a family of metrics $(g_t)_{ab} = (1+ft) g_{ab}$. But then $g_t$ is in the same conformal class as $g$, and hence any such variation must vanish from the conformal invariance of the integral. Hence, $\bar{O}_{(i)}$ vanishes for each $i$; indeed, going forward any coefficient of a pure-trace term must vanish.

Now,  from Proposition~\ref{normal-operators},
we may express the various projections of normal derivatives on the weight $2$ tensor $h$ in terms of conformally-invariant operators on $\mathring{h}$. In particular, starting at $i = d-3$, we may express
\begin{align*}
n^a n^b \nabla_n^{d-3} \mathring{h}_{ab} &= (\delta^{(d-3)}_{nn} \mathring{h}) + \text{lower order terms} \\
\top n^b \nabla_n^{d-3} \mathring{h}_{ab} &= (\delta^{(d-3)}_{n} \mathring{h})_a + \text{lower order terms} \\
\otop \nabla_n^{d-3} \mathring{h}_{ab} &= (\delta^{(d-3)}_{\otop} \mathring{h})_{ab} + \text{lower order terms}\,.
\end{align*}
Thus, we can replace the $i=d-3$ terms in the above integral with conformally-invariant expressions plus lower order terms that depend on at most $d-4$ transverse derivatives of $h$. In doing so, it is possible that there may be intrinsic hypersurface derivatives on those lower order terms. To handle those, we can integrate-by-parts, leaving us with a new (but equivalent) integral. Then, by induction we can apply this same procedure to the case where $i = d-4$, replacing each normal derivative with the conformally-invariant equivalent. Descending down to $i = 2$ for $\delta^{(i)}_{n}$ and $i = 1$ for the other operators, we come to an  the integral of the form
\begin{multline} \label{boundary-term}
\int_{\Sigma} dV_{\bar g} \sum_{i=1}^{d-3} \left( \tilde{P}_{(i)} (\delta^{(i)}_{nn} \mathring{h})+ \tilde{R}_{(i)}^{ab} (\delta^{(i)}_{\otop} \mathring{h})_{ab} \right) + \int_{\Sigma} dV_{\bar g}  \sum_{i=2}^{d-3} \tilde{Q}_{(i)}^a (\delta^{(i)}_n \mathring{h})_a
\\+ \int_{\Sigma} dV_{\bar g} \left(  \hat{P}_{(0)} \mathring{h}_{nn} + \hat{Q}_{(0)}^a \mathring{h}_{na}^\top +  \hat{Q}_{(1)}^a \top (n^b \nabla_n \mathring{h}_{ab} )+ \hat{R}_{(0)}^{ab} \otop \mathring{h}_{ab} \right)\,.
\end{multline}

As the whole integral is conformally invariant and each term in the decomposition of the transverse jets of the variation is independent and via arbitrary variations maps surjectively onto its range, it follows from Proposition~(\ref{normal-operators} that each of the tensors  $\tilde{P}_{(i)}$ and $\tilde{R}_{(i)}^{ab}$ for $1 \leq i \leq d-3$ is conformally-invariant and for each $2 \leq i \leq d-3$, $\tilde{Q}_{(i)}^a$ is conformally invariant. So each term in the first two integrals has been decomposed into pointwise conformal invariants, and the third integral in Display~(\ref{boundary-term}) is independently conformally invariant.

Now observe that one may choose $h$ such that $n^b \nabla_n \mathring{h}_{ab}|_{\Sigma} \neq 0$ but $\mathring{h}|_{\Sigma} = 0$. Then, $n^b \nabla_n \mathring{h}_{ab}|_{\Sigma}$ is  itself conformally invariant, as its failure to be conformally invariant vanishes when $\mathring{h}|_{\Sigma}$ vanishes. As the coefficient $n^b \nabla_n \mathring{h}_{ab}|_{\Sigma}$ is arbitrary, this ensures that $\hat{Q}^a_{(1)}$ must be pointwise conformally invariant and is a section of $T\Sigma[1-d]$.

Now, by explicit computation we find that under a conformal rescaling $g \mapsto \Omega^2 g$,
$$\top (n^b \nabla_n \mathring{h}_{ab}) \mapsto \top (n^b \nabla_n \mathring{h}_{ab})  - \bar{\Upsilon}_a \mathring{h}_{nn} + \top (\bar{\Upsilon}^c \mathring{h}_{ac})\,,$$
where $\bar{\Upsilon} = \nabla^\top \log \Omega |_{\Sigma}$. Thus, by conformally varying each term in the third integral of Display~(\ref{boundary-term}) and recalling that said integral is itself conformally invariant, it is clear that
\begin{align*}
\hat{P}_{(0)} \mapsto& \bar{\Omega}^{1-d}(\hat{P}_{(0)} + \bar{\Upsilon}_a \hat{Q}^{a}_{(1)} )\\
\hat{R}^{ab}_{(0)} \mapsto& \bar{\Omega}^{-1-d}(\hat{R}^{ab}_{(0)} - \bar{\Upsilon}^{(a} \hat{Q}^{b)_{\circ}}_{(1)})\,,
\end{align*}
where $\bar{\Omega} := \Omega|_{\Sigma}$.

Observe that $\bar{\nabla}_a \hat{Q}^a_{(1)}$ is conformally invariant and
$$\bar{\nabla}^{(a} \hat{Q}^{b)_{\circ}}_{(1)} \mapsto \bar{\Omega}^{-1-d}(\bar{\nabla}^{(a} \hat{Q}^{b)_{\circ}}_{(1)} - (d-1) \bar{\Upsilon}^{(a} \hat{Q}^{b)_{\circ}}_{(1)})\,.$$
So, defining
\begin{align*}
\tilde{Q}_{(0)}^a :=& \hat{Q}_{(0)}^a \\
\tilde{R}_{(0)}^{ab} :=& \hat{R}_{(0)}^{ab} - \tfrac{1}{d-1} \bar{\nabla}^{(a} \hat{Q}^{b)_{\circ}}_{(1)}\,,
\end{align*}
each of $\tilde{Q}^a_{(0)}$ and $\tilde{R}^{ab}_{(0)}$ are pointwise conformally invariant. Furthermore, the third integral in Display~(\ref{boundary-term}) becomes
\begin{multline*}
\int_{\Sigma} dV_{\bar g} \left(  \tilde{Q}_{(0)}^a \mathring{h}_{na}^\top + \tilde{R}_{(0)}^{ab} \otop \mathring{h}_{ab} \right) \\
+\int_{\Sigma} dV_{\bar g} \left( \hat{P}_{(0)} \mathring{h}_{nn}  + \hat{Q}^a_{(1)} \top (n^b \nabla_n \mathring{h}_{ab}) + \tfrac{1}{d-1} (\bar{\nabla}^{(a} \hat{Q}^{b)_{\circ}}_{(1)})\otop  \mathring{h}_{ab}  \right)\,.
\end{multline*}

%

Finally we note that the leading transverse-order term in each of $\bar{Q}_{(i)}$ and $\bar{R}_{(i)}$ is unchanged as it is replaced by $\tilde{Q}_{(i)}$ and $\tilde{R}_{(i)}$ respectively. This follows because the leading transverse order term of any coefficient with index $(i)$ is greater than the leading transverse order term of any coefficient with index $(j)$ when $i < j$.
So, $\tilde{R}^{ab}_{(0)} \in \Gamma(T^* \Sigma[3-d])$ and it has leading transverse order term $ {\otop} \left(n^d \Delta^{\frac{d-4}{2}} C_{d(ab)} \right)$, which has transverse order $d-1$. Thus the theorem holds.
\end{proof}

\begin{remark}
For notational brevity going forward, we will define $\delta^{(1)}_n \equiv 0$, the zero operator, $\tilde{Q}^a_{(1)} \equiv 0$, the zero section of $T \Sigma[1-d]$, and $\tilde{P}_{(0)} \equiv 0$, the zero section of $\ce \Sigma[1-d]$. Further, we will denote by $I$ the conformally-invariant integral
$$I := \int_{\Sigma} dV_{\bar g} \left( \hat{P}_{(0)} \mathring{h}_{nn}  + \hat{Q}^a_{(1)} n^b \nabla_n \mathring{h}_{ab} + \tfrac{1}{d-1} (\bar{\nabla}^{(a} \hat{Q}^{b)_{\circ}}_{(1)}) \otop \mathring{h}_{ab} \right)\,,$$
so we can write the variation $\delta S$ as
\begin{align} \label{variational-formula}
\delta S = \int_M \mathcal{O}^{(d)}_{ab} h^{ab} dV_g + I + 
\int_{\Sigma} dV_{\bar g} \sum_{i=0}^{d-3} \left(\tilde{P}_{(i)} (\delta^{(i)}_{nn} \mathring{h})+  \tilde{Q}^a_{(i)} (\delta_n^{(i)} \mathring{h})_a +  \tilde{R}_{(i)}^{ab} (\delta^{(i)}_{\otop} \mathring{h})_{ab} \right)\,.
\end{align}
In the above, each $\tilde{R}_{(i)}$ is an example of a $\FF{d-i}$, and we suspect that $\tilde{P}_{(i)}$ are closely related to the scalar conformal fundamental forms of~\cite{GoverKopWaldron}. It appears that each $\tilde{Q}_{(i)}^a$ is a novel vector-valued conformal hypersurface invariant.
\end{remark}


Theorem~\ref{DN} arises as a corollary of Theorem~\ref{main}:
\begin{proof}[Proof of Theorem~\ref{DN}]
In Theorem~\ref{main} we have established, among other things, that there exists a natural conformal hypersurface invariant $\FF{d} \in \Gamma(\odot^2_\circ T^* \Sigma[3-d])$ with transverse order $d-1$ and leading term $\otop \left(n^d \Delta^{\frac{d-4}{2}} C_{d(ab)} \right)$.  From~\cite[Corollary 3.3]{blitz-classifying}, it follows that this section has unique leading transverse order term. In general, such an invariant is not unique as it may be adjusted (at subleading order) by adding conformally-invariant terms of the appropriate weight and tensor type of lower transverse order like $(\IIo^{d-1})_{(ab)_{\circ}}$. However, when $\Sigma \hookrightarrow (M,\cc)$ is the conformal manifold associated with a Poincar\'e--Einstein manifold, it follows from~\cite[Theorem 1.8]{BGW1} and~\cite[Theorem 1.1]{blitz-classifying} that all subleading conformally-invariant terms vanish, with the possible exception of conformal invariants intrinsic to the hypersurface $(\Sigma,\bar{\cc})$. However, there are no such conformal invariants of weight $3-d$, as $3-d$ is odd and all natural intrinsic conformal invariants have even weight. So the uniqueness of the leading order structure and the vanishing of subleading conformal invariants implies that this invariant is uniquely determined by the Poincar\'e--Einstein structure (up to multiplication by a constant). But then it must agree with the image of the Dirichlet-to-Neumann map $\operatorname{DN}^{(d)}$ as described in~\cite{BGKW}.
\end{proof}

\subsection{Conserved Currents}
Having constructed $\FF{d}$, we now may use it to further glean information about the way the boundary of a conformal manifold is embedded. Indeed, we may straightforwardly prove Corollary~\ref{divergence-corollary}.
\begin{proof}[Proof of Corollary~\ref{divergence-corollary}]
First, observe that by a straightforward computation, one may show that for any section $t_{ab} \in \odot^2_{\circ} T^*\Sigma[3-d]$, it follows that $\bar{\nabla}^a t_{ab}$ is itself conformally invariant and a section of $T^* \Sigma[1-d]$.

Next, from Theorem~\ref{DN}, it follows that $\FF{d}$ is the image of the Dirichlet-to-Neumann map for the Poincar\'e--Einstein problem. But then its divergence vanishes, see~\cite[Remark 3.10]{BGKW}.
\end{proof}

In light of Corollary~\ref{divergence-corollary}, we may improve the
definition found in~\cite{BGW1} for an asymptotically
Poincar\'e--Einstein manifold. To do so we first requires a standard definition:
\begin{definition} 
Let $(M,\cc)$ be a conformal manifold with boundary. Recall $g^o \in
\cc|_{M_+}$ is conformally compact if $g_s=s^2 g^o$ extends smoothly
as a metric to the boundary, for $s \in C^\infty M$ some boundary
(equivalently any) defining function. In this case the scalar
curvature $Sc^{g^o}$ extends smoothly to the boundary, and we use
$Sc^{g^o}$ to denote that extension.  Let $g^o \in \cc|_{M_+}$ be a
conformally compact metric such that
$$Sc^{g^o} = -d(d+1) + \mathcal{O}(s)\, .$$   We call such a metric representative an
\textit{asymptotically hyperbolic} metric.
\end{definition}

Similarly, from the conformal invariance of the almost-Einstein
operator \eqref{Eop} it follows that, on a conformally compact
manifold $(M,g^o)$, the quantity $s \mathring{P}^{g^o}$ (with $s$ as
above) also extends smoothly to the boundary. As the vanishing of this
quantity is equivalent to the Einstein condition, we are now able to
define a special kind of asymptotically hyperbolic metric:
\begin{definition}\label{asymp-PE}
  Let $(M,\cc)$ be a conformal manifold with boundary and let $g^o$ be an asymptotically hyperbolic metric on its interior. Then $g^o$ is said to be \textit{asymptotically Poincar\'e--Einstein} if
  $$s\mathring{P}^{g^o} = \mathcal{O}(s^{d-2}) ,$$
  where $s \in C^\infty M$ that is a defining function for the boundary, 
and  the boundary embedding $\partial M \hookrightarrow (M,\cc)$ satisfies $\bar{\nabla} \cdot \FF{d} = 0$.
\end{definition}
\noindent Observe that the first condition for a metric $g^o$ to be
asymptotically Poincar\'e--Einstein is directly inspired. However the second
condition can not be easily read off from the Einstein condition
alone, but is rather suggested by  the result in Corollary~\ref{divergence-corollary}.

In view of Definition~\ref{asymp-PE} and Propositions~\ref{singyam}
and~\ref{unique-singyam}, it is evident that when there exists an
asymptotically Poincar\'e--Einstein metric in the conformal class of
the interior of a conformal manifold with boundary, the singular
Yamabe metric is one such metric.  In fact it is {\em the} canonical
such metric, up to the given order, as a consequence of the uniqueness
of the singular Yamabe solution. (The point here is that any PE metric
is, in particular, a singular Yamabe solution, so if a PE metric exists, the uniqueness of the singular Yamabe metric implies the PE metric is also unique.) For brevity, we will
say that a conformal manifold with boundary admits an asymptotically
Poincar\'e--Einstein metric when there exists such a metric in the
conformal class of the interior of the manifold.
%
%

As a consequence of the main theorem of~\cite{BGW1} and the Definition~\ref{asymp-PE} given above, we are thus able to state the simple corollary:
\begin{corollary} \label{APE-CFFs}
Let $(M,\cc)$ be an even-dimensional conformal manifold with boundary. Then $(M,\cc)$ admits an asymptotically Poincar\'e--Einstein metric if and only if $\FF{k} = 0$ for all $k \in \{2, 3, \ldots, d-1\}$ and $\bar{\nabla} \cdot \FF{d} = 0$.
\end{corollary}

On the other hand, when $\bar{\nabla} \cdot \FF{d} \neq 0$, this 1-form is a distinguished 1-form on the boundary of the conformal $d$-manifold.

\section{Conformally-Invariant Variational Problems on Manifolds with Boundary}
\label{boundary-problems}

As a consequence of the construction used in Section~\ref{theorem-proof}, we find applications in, among other things, variational approaches to the Fefferman--Graham flat problem, meaning the boundary problem associated to the PDE $\mathcal{O}^{(d)} = 0$. To see this, consider the classical case of the Einstein--Hilbert action on a 4-manifold:
$$S_{EH} = \tfrac{1}{16\pi} \int_M Sc \, dV_g\,,$$
where $Sc$ is the scalar curvature. On a manifold with boundary, it is known that~\cite{bavera}
\begin{align} \label{variation-EH}
16 \pi \delta S_{EH} = -\int_M G_{ab} h^{ab} dV_g + \int_{\partial M} (2 H \bar{g}^{ab} - \IIo^{ab}) h^\top_{ab} dV_{\bar g} - 6 \int_{\partial M} \delta (H dV_{\bar g})\,.
\end{align}
On a manifold without boundary, solving the problem $\delta S_{EH} = 0$ is equivalent to solving the Einstein field equations $G_{ab} = 0$ in a vacuum. Furthermore, it is known that the partial differential equation given by $G_{ab} = 0$, for metrics with Euclidean signature, is, after gauge-fixing, elliptic~\cite{Besse1987}. Now because $\delta H$ depends on the 1-jet of $h$ along $\Sigma$, and we may choose $h$ arbitrarily, it is clear that all three of the integrals must vanish independently to ensure that $\delta S_{EH} = 0$. But enforcing this then amounts to specifying both $g|_{ \partial M}$ and its first (normal) derivative. (Note that one may alternatively fix the coefficient of $h^\top$ in the second integral to vanish---this does not change the claim here.) In contrast, strictly fewer boundary conditions are required to make the gauge-fixed Einstein problem an elliptic boundary problem (of Fredholm index $0$)~\cite{Anderson2008}. So, we find that that, as written, the variational version of the boundary-value problem is overdetermined.

To resolve this problem, Gibbons, Hawking, and York modified the Einstein--Hilbert action. Consider the energy functional
$$S = \tfrac{1}{16\pi} \int_M Sc \, dV_{g} + \tfrac{3}{8\pi} \int_{\partial M} H dV_{\bar{g}}\,,$$
Varying this functional, we have that
$$16 \pi \delta S = -\int_M \sqrt{g} G_{ab} \delta g^{ab} dV_g + \int_{\partial M}  (2 H \bar{g}_{ab} - \IIo_{ab}) \delta \bar{g}^{ab} dV_{\bar{g}}\,,$$
Comparing with~\cite{Anderson2008}, it is now clear that solving the variational system is at least not overdetermined (in fact, it is underdetermined as gauge-fixing is not required to solve the variational problem). Alternatively, one may instead prescribe that $\II = 0$ and find the same. That is, the (no longer overdetermined) variational problem requires that either $\II = 0$ or that we fix $\bar{g}$.

Note that this procedure picks out an important ordered pair of extrinsic tensors for the Einstein equation on a manifold with boundary: the induced metric $\bar{g}$ and a trace-corrected form of the second fundamental form $\II$, $(\bar{g}, \II - 3 H \bar{g})$. Note that this is \textit{not} the trace-free part of $\II$.

For later comparison, it is helpful to record the variation of the Einstein--Hilbert action in terms of another set of variables. Unpacking Equation~(\ref{variation-EH}), we find that
\begin{multline*}
16 \pi \delta S_{EH} = -\int_M G_{ab} \delta g^{ab} dV_g \\+ \int_{\partial M} \left[(- \delta g)_{ab}^\top (\IIo^{ab} + H \bar{g}^{ab}) + 3 (\delta g)_{nn} H - (\nabla_n \delta g_{ab}) \bar{g}^{ab}  \right] dV_{\bar{g}}\,.
\end{multline*}
Observe that in this form, the overdetermined nature of the system is more manifest: without the imposition of extrinsic constraints, one would need both $\delta g|_{\partial M}$ and $\nabla_n \delta g|_{\partial M}$ to vanish as they are independent. In this form, the natural ordered pair of the Dirichlet data $\bar{g}$ and the Neumann data $\II - 3 H \bar{g}$ is less evident. However, one can still infer that, if $\delta g|_{\partial M}$ is to be free, then the second fundamental form  must vanish.

\medskip

Consider now the variational form of the Fefferman--Graham flat
problem, as discussed in Proposition~\ref{obstruction-variation},
except we now wish to treat a manifold with boundary. From
Equation~(\ref{variational-formula}), it is clear that the boundary
terms would depend non-trivially on both the variation of $g$ along
the boundary \textit{as well as} normal derivatives of that
variation. On the other hand the equation $\mathcal{O}^{(d)}=0$ is
elliptic of order $d$, up to the need for gauge fixing due to
diffeomorphism invariance and conformal invariance
\cite{A-Via}.
Thus one should not need to specify more than $d/2$ sections of $\odot^2 T^*
\partial M$.
Interestingly, when $d=4$, for the Fefferman--Graham flat
problem (i.e.\ the \textit{Bach flat problem}), we see that indeed only
two sections of $\odot^2 T^* \partial M$ must be fixed to force the
boundary terms to vanish. On the other hand, for all even $d \geq 6$,
the system appears to become ``overdetermined''. This suggests that one
should modify the energy functional by a boundary term in the spirit of
the Gibbons--Hawking--York boundary term. To see how this happens, we
first prove the following technical lemma:
\begin{lemma} \label{IIo-variation}
Let $\Sigma \hookrightarrow (M,g)$ be a smooth hypersurface embedding. Then, with respect to metric variations, the trace-free second fundamental form of $\Sigma$ varies according to
$$\delta \IIo_{ab}  =  \tfrac{1}{2} \delta_R \mathring{h}_{ab}  + 2 n_{(a}  \IIo^c_{b)} h_{nc}^\top  - \tfrac{1}{2} \IIo_{ab} h_{nn} + \IIo_{(a}^c  h^\top_{b)c}\,.$$ 
\end{lemma}
\begin{proof}
To compute the variation of the trace-free fundamental form, we use the identity
$$\II_{ab} \= \bar{g}_{a}^c \bar{g}_{b}^d \nabla_c \frac{n_d}{\sqrt{|n|^2_g}}\,.$$
Note that for a given metric $g_0$, we may always choose for $\Sigma$ a defining function $s$ such that $n = ds$, with $|n|^2_g = 1$ (see e.g.~\cite{Wald}). It is important to note that, in the variations that follow, we treat $s$ as fixed (such that $ds$ has unit length in the reference metric), but $\delta(g^{-1}(ds,ds)) \not \= 0$ in general, as we allow the metric to change.

Now, we must find the variations of $\bar{g}_a^c$, $\nabla$, and $\tfrac{1}{\sqrt{|n|^2_g}}$ (viewed as quantities along $\Sigma$) with respect to $g$.
As $\sqrt{g_0^{-1}(n,n)} = 1$, we have that
$$\delta \left(\tfrac{1}{\sqrt{|n|^2_g}}\right) \= -\tfrac{1}{2} (\delta g^{-1})(n,n) \= \tfrac{1}{2} h_{nn}\,.$$
Next, note that
\begin{align*}
\delta \bar{g}_a^b &= \delta g_a^b - \delta(\tfrac{1}{|n|^2_g} n_a g^{bc} n_c) \\
&= - h_{nn}  n_a n^b + n_a h^{b}_n\,.
\end{align*}
Finally, we use variation of the connection, which is a standard result:
$$(\delta \nabla)_a v_b = -C^c_{ab} v_c\,,$$
where
\begin{align} \label{variation-nabla2}
C^a_{bc} = \tfrac{1}{2} g^{ad}( \nabla_b h_{dc} + \nabla_c h_{bc} - \nabla_d h_{bc})\,.
\end{align}

Putting all this together and using that $|n|^2_{g_0} = 1$, we have that
\begin{align*}
\delta \II_{ab} \=&  (-h_{nn} n_a n^c + n_a h_n^c) \bar{g}_b^d \nabla_c n_d + \bar{g}_a^c (-h_{nn} n_b n^d + n_b h_n^d) \nabla_c n_d \\&- \tfrac{1}{2} \bar{g}_a^c \bar{g}_b^d n^e (\nabla_c h_{de} + \nabla_{b} h_{de} - \nabla_{e} h_{dc}) + \tfrac{1}{2} \bar{g}_a^c \bar{g}_b^d\nabla_c (h_{nn} n_d) \\
\=& n_a h_n^c \II_{cb}+ n_b h_n^c \II_{ca} - \tfrac{1}{2} \bar{g}_a^c \bar{g}_b^d n^e (\nabla_c h_{de} + \nabla_{d} h_{ce} - \nabla_{e} h_{cd}) + \tfrac{1}{2} \II_{ab} h_{nn}\\
\=& 2 n_{(a}  \II^c_{b)} h_{nc}^\top - \bar{g}_{(a}^d \nabla_{b)}^\top (h_{dn}^\top + n_d h_{nn}) + \II_{(a}^c  h^\top_{b)c} + \tfrac{1}{2} \bar{g}_a^c \bar{g}_b^d \nabla_n h_{cd} + \tfrac{1}{2} \II_{ab} h_{nn} \\
\=& 2 n_{(a}  \II^c_{b)} h_{nc}^\top - \bar{g}_{(a}^d \nabla_{b)}^\top h_{dn}^\top - \tfrac{1}{2}\II_{ab} h_{nn} + \II_{(a}^c  h^\top_{b)c} + \tfrac{1}{2} \bar{g}_a^c \bar{g}_b^d \nabla_n h_{cd} \\
\=& 2 n_{(a}  \II^c_{b)} h_{nc}^\top - \bar{\nabla}_{(a} h_{b)n}^\top - \tfrac{1}{2}\II_{ab} h_{nn} + \II_{(a}^c  h^\top_{b)c} + \tfrac{1}{2} \bar{g}_a^c \bar{g}_b^d \nabla_n h_{cd} \\
\=& 2 n_{(a}  \IIo^c_{b)} h_{nc}^\top  - \tfrac{1}{2} \IIo_{ab} h_{nn} + \IIo_{(a}^c  h^\top_{b)c} -\bar{\nabla}_{(a} h_{b)n}^\top + \tfrac{1}{2} \bar{g}_a^c \bar{g}_b^d \nabla_n \mathring{h}_{cd}   \\&+ 2 H n_{(a}  h_{b)n}^\top - \tfrac{1}{2} H \bar{g}_{ab} h_{nn} + H h^\top_{ab} + \tfrac{1}{2d} \bar{g}_{ab}\nabla_n h_{c}^c\,.
\end{align*}
In the second equality, we have used the fact that $\nabla n = \II$ and that $n \cdot \II = 0$. In the third equality, we have expressed covariant derivatives in terms of the tangential and normal parts, i.e. $\nabla = \nabla^\top + n \nabla_n$, and then expressed $h_{na}$ in terms of its projections along $\Sigma$. In the fourth equality, we have combined terms and simplified, using that $\nabla n = \II$ as before. In the fifth equality, we have used the Gauss formula, namely that for any $v \in \Gamma(T^* \Sigma)$, we have that $\nabla_a^\top v_b = \bar{\nabla}_a v_b - n_b \II_a^c v_c$, where $\nabla_a^\top v_b$ is defined using any extension of $v$ to $T^*M$. In the final equality, we have expressed $\nabla_n h$ in terms of $\nabla_n \mathring{h}$ and all instances of $\II$ in terms of $\IIo$.

Now, to compute the variation of the trace-free second fundamental form, recall that
$$\IIo_{ab} = \II_{ab} - \tfrac{1}{d-1} \bar{g}_{ab} \bar{g}^{cd} \II_{cd}\,.$$
To compute its variation, note that $\bar{g} = g - \tfrac{1}{|n|^2_g} n \otimes n$, so
$$\delta \bar{g}_{ab} = h_{ab} + \delta (|n|^2_g) n_a n_b = h_{ab} - n_a n_b h_{nn}\,.$$
Because $\bar{g}^{ab} \bar{g}_{bc} = \bar{g}_c^a$, we also have that $\delta \bar{g}^{ab} = -\bar{g}^{ac} \bar{g}^{bd} h_{cd}$. Putting these together, we have that
\begin{align*}
\delta \IIo_{ab} \=& \delta \II_{ab} - \tfrac{1}{d-1} \left[ (d-1)(h_{ab} - n_a n_b h_{nn}) H - \bar{g}_{ab} \II^{cd} h_{cd} + \bar{g}_{ab} \bar{g}^{cd} \delta \II_{cd} \right] \\
\=& (\delta \II)_{(ab)_{\circ}} - (h_{ab} - n_a n_b h_{nn}) H + \tfrac{1}{d-1} \bar{g}_{ab} \II^{cd} h_{cd}\,.
\end{align*}

We can now finish the computation:
\begin{align*}
\delta \IIo_{ab} \=& 2 n_{(a}  \IIo^c_{b)} h_{nc}^\top  - \tfrac{1}{2} \IIo_{ab} h_{nn} + \IIo_{(a}^c  h^\top_{b)c} - \tfrac{1}{d-1} \bar{g}_{ab} \IIo^{cd} h_{cd} -\bar{\nabla}_{(a} h_{b)_{\circ} n}^\top + \tfrac{1}{2} \otop \nabla_n \mathring{h}_{cd}   \\&+ 2 H n_{(a}  h_{b)n}^\top + H h^\top_{ab} - \tfrac{1}{d-1} H \bar{g}_{ab} \bar{g}^{cd} h_{cd} - H h_{ab} + n_a n_b H h_{nn} + \tfrac{1}{d-1} \bar{g}_{ab} \II^{cd} h_{cd} \\
 \=& 2 n_{(a}  \IIo^c_{b)} h_{nc}^\top  - \tfrac{1}{2} \IIo_{ab} h_{nn} + \IIo_{(a}^c  h^\top_{b)c} -\bar{\nabla}_{(a} h_{b)_{\circ} n}^\top + \tfrac{1}{2} \otop \nabla_n \mathring{h}_{cd}   \\
 \=& 2 n_{(a}  \IIo^c_{b)} h_{nc}^\top  - \tfrac{1}{2} \IIo_{ab} h_{nn} + \IIo_{(a}^c  h^\top_{b)c}+ \tfrac{1}{2} \delta_{\otop}^{(1)} \mathring{h}_{cd} \,.
\end{align*}
In the third equality above, we have used Equation~(\ref{dn-otop}). So the lemma follows.
\end{proof}

\bigskip

Now, for the sake of illustrating the behavior of the boundary integrals, consider the action on a 4-manifold with boundary that varies in the Bach tensor in the interior. We apply the variation algorithm outlined in the proof of Theorem~\ref{main} to this case.
\begin{proposition} \label{Weyl-squared}
Let $(M^4,g)$ be a compact manifold with boundary $\partial M$. Then, the functional variation of $S = \frac14 \int_M |W|^2 \, dV$ with respect to metric variations is given by
\begin{align*}
\delta S = &-\int_M B^{ab} h_{ab} dV_g \\&+\int_{\partial M} \big[\tfrac{4}{3} W_n{}^{bc}{}_n \IIo_{bc} \mathring{h}_{nn} -(\IVo^{ab} + \IIo_c^{(a} W_{n}{}^{b)_{\circ} c}{}_n) \mathring{h}_{ab}^{\otop} - W_n{}^{bc}{}_{n} \delta^{(1)}_{\otop} \mathring{h}_{bc} \big] dV_{\bar g}\,.
\end{align*}
where $\delta g = h$.
\end{proposition}
\begin{proof}
In order to use more standard variation formulas, we first translate the action into an integral of Riemann tensors and its traces:
$$S = \int_M \left(\tfrac{1}{4} R_{abcd} R^{abcd} - \tfrac{1}{2} Ric_{ab} Ric^{ab} + \tfrac{1}{12} Sc^2 \right) dV\,.$$
Now, defining $h = \delta g$ and using the (standard result for the) variation of the covariant derivative $(\delta \nabla)_a v_b = -\gamma^c_{ab} v_c$, where $\gamma$ is defined according to
\begin{align} \label{variation-nabla}
\gamma^a_{bc} = \tfrac{1}{2} g^{ad}( \nabla_b h_{dc} + \nabla_c h_{bd} - \nabla_d h_{bc})\,,
\end{align}
we may compute the variation of the Riemann tensor (see~\cite[Section 7.5]{Wald}). To do so, let $v \in \Gamma(T^* M)$.
\begin{align*}
(\delta R_{abc}{}^d) v_d =& \delta (\nabla_a \nabla_b v_c - \nabla_b \nabla_a v_c) \\
=& (\delta \nabla_a) \nabla_b v_c + \nabla_a (\delta \nabla_b) v_c - (\delta \nabla_b) \nabla_a v_c - \nabla_b (\delta \nabla_a) v_c \\
=& -\gamma^d_{ab} \nabla_d v_c - \gamma^d_{ac} \nabla_b v_d - \nabla_a (\gamma^d_{bc} v_d) + \gamma^d_{ba} \nabla_d v_c + \gamma^d_{bc} \nabla_a v_d + \nabla_b (\gamma^d_{ac} v_d) \\
=& - \gamma^d_{ac} \nabla_b v_d - \nabla_a (\gamma^d_{bc} v_d) + \gamma^d_{bc} \nabla_a v_d + \nabla_b (\gamma^d_{ac} v_d) \\
=& \left[(\nabla_b \gamma^d_{ac}) -(\nabla_a \gamma^d_{bc}) \right] v_d \\
=& \tfrac{1}{2} g^{de} \left[ \nabla_b (\nabla_a h_{ce} + \nabla_c h_{ae} - \nabla_e h_{ac}) - \nabla_a (\nabla_b h_{ce} + \nabla_c h_{be} - \nabla_e h_{bc}) \right] v_d \\
=& \tfrac{1}{2} g^{de} (\nabla_b \nabla_c h_{ae} - \nabla_b \nabla_e h_{ac} - \nabla_a \nabla_c h_{be} + \nabla_a \nabla_e h_{bc}) v_d \\&+ \tfrac{1}{2} g^{de} (R_{bac}{}^f h_{fe} + R_{bae}{}^f h_{cf}) v_d\,.
\end{align*}
So, it follows that
$$\delta R_{abcd} = \tfrac{1}{2}(\nabla_b \nabla_c h_{ad} - \nabla_b \nabla_d h_{ac} - \nabla_a \nabla_c h_{bd} + \nabla_a \nabla_d h_{bc}) + \tfrac{1}{2} R_{abc}{}^e h_{ed} - \tfrac{1}{2} R_{abd}{}^e h_{ce} \,.$$

Using this formula and its traces, we come to the variation formula
$$\delta S = \int_M \left(W^{cabd} \nabla_a \nabla_b h_{cd} - \tfrac{1}{2} W^{acde} W^b{}_{cde} h_{ab} + \tfrac{1}{8} |W|^2 h^a_a - W^{abcd} P_{bd} h_{ac} \right) dV_g \,.$$
Integrating by parts to isolate $h_{cd}$ and applying the divergence theorem (noting that $n$ is inward pointing), we have that
$$\delta S = -\int_M B^{ab} h_{ab} dV_g + \mathcal{B} \,,$$
where
$$\mathcal{B} := - \int_{\partial M} \left(n^c C_{cab} h^{ab} - n_a W^{abcd} \nabla_c h_{bd} \right) dV_{\bar g}\,.$$
Note that the integrand of $\mathcal{B}$ is pointwise conformally-invariant.

We can now handle the boundary integral $\mathcal{B}$. First, rewrite $\nabla_c h_{bd} = (\nabla_c^\top  + n_c \nabla_n) h_{bd}$ and then express
$$h_{ab} \= h_{ab}^\top + n_a h_{nb}^\top + n_b h_{na}^\top + n_a n_b h_{nn}\, ,$$
yielding
$$\mathcal{B} = \int_{\partial M} \left[-n^c C_{cab} h^{ab} + n_a W^{abcd} \nabla_c^\top (h_{bd}^\top + 2n_{(b} h_{d)n}^\top + n_b n_d h_{nn})  + n_a W^{abcd} n_c \nabla_n h_{bd} \right] dV_{\bar g}\,.$$
Using the fact that $\nabla^\top n = \II$, we can write
\begin{align*}
\mathcal{B} = \int_{\partial M} \big[&-n^c C_{cab} h^{ab} + n_a W^{abcd} \nabla_c^\top h_{bd}^\top + 2 n_a W^{abcd} \II_{c(b} h_{d)n}^\top \\&+ n_a n_d W^{abcd} \nabla_c^\top h_{bn}^\top + n_a n_d W^{abcd} \II_{bc} h_{nn} + n_a W_{abcd} n_c \nabla_n h_{bd} \big] dV_{\bar g}\,.
\end{align*}
Now using the Gauss formula
$$\nabla_a^\top v_b = \bar{\nabla}_a v_b - n_b \II_a^c v_c\,,$$
where $v \in \Gamma(T^* \Sigma)$ and $\nabla^\top v$ is defined using any extension of $v$ to $T^* M$, we have that we have that
\begin{align*}
\mathcal{B} = \int_{\partial M} \big[&-n^c C_{cab} h^{ab} + n_a W^{abcd} (\bar{\nabla}_c h_{bd}^\top - n_d \II_c^e h_{be}^\top) + 2 n_a W^{abcd} \II_{c(b} h_{d)n}^\top \\&+ W_n{}^{bc}{}_{n} \bar{\nabla}_{b} h_{cn}^\top + W_n{}^{bc}{}_n \IIo_{bc} h_{nn} - W_n{}^{bc}{}_{n} \nabla_n h_{bc} \big] dV_{\bar g}\,.
\end{align*}
Now, using the known expression~\cite{BGW1} for 
\begin{align} \label{dn-otop}
\delta^{(1)}_{\otop} u_{ab} := \otop [ \nabla_n u_{ab} - 2 \bar{\nabla}_{(a} u_{n b)_\circ}^\top ]\,,
\end{align}
where $u \in \Gamma(\odot^2_\circ T^* M[2])$, we may write
\begin{align*}
\mathcal{B} = \int_{\partial M} \big[&-n^c C_{cab} h^{ab} + n_a W^{abcd} (\bar{\nabla}_c h_{bd}^\top - n_d \II_c^e h_{be}^\top) + 2 n_a W^{abcd} \II_{c(b} h_{d)n}^\top \\&- W_n{}^{bc}{}_{n} \bar{\nabla}_{b} h_{cn}^\top + W_n{}^{bc}{}_n \IIo_{bc} h_{nn} - W_n{}^{bc}{}_{n} \delta_{\otop}^{(1)} \mathring{h}_{bc} \big] dV_{\bar g}\,.
\end{align*}
Now integrating by parts along $\partial M$, we have that
\begin{align*}
\mathcal{B} = \int_{\partial M} \big[&-n^c C_{cab} h^{ab} + (\bar{\nabla}_c  W^{cdb}{}_n^\top  ) h_{bd}^\top - W_n{}^{bc}{}_n \II_c^e h_{be}^\top + 2 n_a W^{abcd} \II_{c(b} h_{d)n}^\top \\&+ (\bar{\nabla}_{b} W_n{}^{bc}{}_{n}) h_{cn}^\top + W_n{}^{bc}{}_n \IIo_{bc} h_{nn} - W_n{}^{bc}{}_{n} \delta_{\otop}^{(1)} \mathring{h}_{bc} \big] dV_{\bar g}\,.
\end{align*}

Recall that for hypersurfaces in four-manifolds, $\IVo_{ab} := C_{n(ab)}^\top + H W_{nabn} - \bar{\nabla}^c W_{c(ab)n}^\top$. We also have $\II = \IIo + H \bar{g}$. So, 
\begin{align*}
\mathcal{B} = \int_{\partial M} \big[&-\IVo_{ab} h^{ab} + C^a{}_{nn} h_{na}^{\top} - W_n{}^{bc}{}_n \IIo_c^e h_{be}^\top  + n_a W^{abcd} \IIo_{cb} h_{dn}^\top  \\&+ (\bar{\nabla}_{b} W_n{}^{bc}{}_{n}) h_{cn}^\top + W_n{}^{bc}{}_n \IIo_{bc} h_{nn} - W_n{}^{bc}{}_{n} \delta_{\otop}^{(1)} \mathring{h}_{bc} \big] dV_{\bar g}\,.
\end{align*}
Using the definition of the Cotton tensor (in terms of the Weyl tensor) and the typical hypersurface identities, one may derive that, in any dimension $d \geq 4$,
\begin{align}
C_{ann} \=& \tfrac{1}{d-3} n^b n^c \nabla^d W_{dcab} \nonumber\\
\=& \tfrac{1}{d-3} n^b n^c \nabla^{\top \, d} W_{dcab} \nonumber \\
\=& -\tfrac{1}{d-3} \nabla^{\top \, d} W_{nadn} + \tfrac{1}{d-3}\IIo^{bd} W_{ndab}  \nonumber\\
\=& -\tfrac{1}{d-3} \bar{\nabla}^d W_{nadn} + \tfrac{1}{d-3} n_a \IIo^{de} W_{nedn}   + \tfrac{1}{d-3}\IIo^{bd} W_{ndab} \nonumber\\
\=& -\tfrac{1}{d-3} \bar{\nabla}^d W_{nadn}   + \tfrac{1}{d-3}\IIo^{bd} W_{ndab}^\top\,. \nonumber
\end{align}
So, we have that
\begin{align} \label{W2-variation-deltaR}
\mathcal{B} = \int_{\partial M} \big[&-(\IVo^{ab} + \IIo_c^{a} W_{n}{}^{bc}{}_n) h_{ab}^\top   + W_n{}^{bc}{}_n \IIo_{bc} h_{nn} - W_n{}^{bc}{}_{n} \delta_{\otop}^{(1)} \mathring{h}_{bc} \big] dV_{\bar g}\,.
\end{align}
Finally, noting that $h_{nn} = n^a n^b h_{ab} = \mathring{h}_{nn} + \tfrac{1}{4} \tr h$ and that
$$h^\top_{ab} = \mathring{h}^{\otop}_{ab} + (\tfrac{1}{4} \tr h - \tfrac{1}{3} \mathring{h}_{nn})\bar{g}_{ab}\,,$$
where $\mathring{h}^{\otop}$ is understood to mean $\otop (\mathring{h})$,
 we may simplify to
\begin{align*} 
\mathcal{B} = \int_{\partial M} \big[&-(\IVo^{ab} + \IIo_c^{(a} W_{n}{}^{b)_{\circ}c}{}_n) \mathring{h}_{ab}^{\otop}   + \tfrac{4}{3} W_n{}^{bc}{}_n \IIo_{bc} \mathring{h}_{nn} - W_n{}^{bc}{}_{n} \delta_{\otop}^{(1)} \mathring{h}_{bc} \big] dV_{\bar g}\,.
\end{align*}
This completes the proof.
\end{proof}

\begin{remark}
Observe that a natural variant of the fourth conformal fundamental form in $d=4$, as suggested by the variational problem in Proposition~\ref{Weyl-squared}, is
$$\IVo_{ab} + \IIo^c_{(a} W_{nb)_\circ cn} =  C_{n(ab)}^\top + H W_{nabn} - \bar{\nabla}^c W_{c(ab)n}^\top + \IIo^c_{(a} W_{|n|b)_\circ cn}\,.$$
\end{remark}

Now as a consequence of Lemma~\ref{IIo-variation} and Proposition~\ref{Weyl-squared} (in particular Equation~\ref{W2-variation-deltaR}), we may re-express the variation of the energy functional $\tfrac{1}{4} \int_ M |W|^2 dV$ on a 4-manifold with boundary as
$$\delta S = -\int_M B^{ab} h_{ab} dV_g \\+ \int_{\partial M} \Big( - [\IVo^{ab} - \IIo_c^{(a} W_{n}{}^{b) c}{}_n] h_{ab} - 2 W_n{}^{ab}{}_n \delta \IIo_{ab}\Big) dV_{\bar{g}}\,.$$
In this form, two natural ordered pairs arise from the Bach flat problem:
$$(\bar{g}_{ab}, -\IVo_{ab} + \IIo^c_{(a} W_{|n|b)cn}) \qquad \text{ and } \qquad (\IIo_{ab}, -2 W_{nabn})\,.$$

\medskip

More generally than Lemma~\ref{IIo-variation}, it is easy to see (from the fact that, in coordinates adapted to the hypersurface embedding, $k$th conformal fundamental forms have leading term $\otop \partial_s^{k-1} g$) that, at leading order,
$$\delta\, \FF{k}_{ab} \propto \delta_{\otop}^{(k-1)} \mathring{h}_{ab} + \text{ lower order terms.}$$
Such a relation allows for two observations. First, it implies a canonical pairing of sections of $\odot^2 T^* \Sigma$ between those with transverse order $m$ and $d-m-1$, expressible in terms of conformal fundamental forms and the induced metric of the form $\FF{m+1} \cdot \delta (\FF{d-m})$. Second, in $d \geq 6$, we may express the boundary contribution to the variation of the Fefferman--Graham action exhibited in Equation~(\ref{variational-formula}) in terms of conformal fundamental forms, their variations, and $I$. As such, it is straightforward to add conformally-invariant integrals of the form found in Display~(\ref{pair2}) to the Fefferman--Graham action to remove the problem of overdeterminedness, when necessary. Such integrals play the same role as the Gibbons-Hawking-York boundary term does.

When such terms are added, the variational problems become (at least) not overdetermined. In that case, we may view the pairings as between Dirichlet data and Neumann data. Because Fefferman--Graham flat problems are elliptic after gauge-fixing, it is sufficient to specify \textit{just} the Dirichlet data to solve the problem globally (or alternatively, one may specify just the Neumann data). At least in the case of the Bach-flat problem where the splitting of the boundary data is clear from the variational problem, when there exists a unique solution the Bach-flat problem on a given 4-manifold with boundary, there exists a well-defined Dirichlet-to-Neumann map, sending
\begin{align} \label{pair-mapping}
(\bar{g}_{ab}, \IIo_{ab}) \mapsto (-2 W_{nabn},-\IVo_{ab} + \IIo^c_{(a} W_{|n|b)cn})\,.
\end{align}
When such a map exists, it is interesting to consider its specializations. This is discussed in Section~\ref{sec:FG-flat}.

\section{Fefferman--Graham Flat Obstructions to Poincar\'e--Einstein} \label{sec:FG-flat}

The uniqueness of Bach-flat manifolds with a given
set of boundary data is not fully understood. It is straightforward to show, and well known, 
that all self-dual, anti-self-dual and Einstein manifolds are Bach-flat.
We are interested here in the link to Poincar\'e-Einstein metrics.  Maldacena
showed~\cite{maldacena} that at the level of small perturbations, when
$\IIo = 0$, solutions to the linearized Bach-flat boundary problem are
also solutions to the linearized Poincar\'e--Einstein
equation. In~\cite{AlaeeWool}, the fully non-linear equations were
considered; they found that an asymptotically hyperbolic metric
satisfies $\IIo = \IIIo = 0 = \bar{\nabla} \cdot \IVo$ (exactly in agreement with Corollary~\ref{APE-CFFs}), has a
vanishing obstruction density, is a formal solution to the Bach-flat
boundary-value problem, and satisfies a fourth constraint not relevant
for our discussion if and only if that asymptotically hyperbolic
metric is a formal solution to the Poincar\'e--Einstein equation.

Inspired by these results, we prove the stronger result that a Bach-flat manifold with umbilic boundary is sufficient to guarantee that the manifold's singular Yamabe metric is a formal solution to the Poincar\'e--Einstein equation.

\begin{theorem} \label{umbilic-PE}
Let $(M^4,\cc)$ be a smooth compact Bach-flat conformal manifold with umbilic boundary. Then the singular Yamabe metric of Proposition~\ref{unique-singyam} is smooth and, formally to all orders, Poincar\'e--Einstein.
\end{theorem} 
\begin{proof}
Let $(M^d,\cc)$ be a Bach-flat conformal manifold with an umbilic boundary. To prove the theorem, we begin by proving that $(M^d,\cc)$ contains an asymptotically Poincar\'e--Einstein metric in the conformal class of its interior. To do this, we rely on Corollary~\ref{APE-CFFs}: that the existence of an asymptotically Poincar\'e--Einstein metric in the conformal class of the interior is equivalent to the conditions that $\Fo = 0$ and that $\bar{\nabla} \csdot \IVo = 0$.

We begin by considering the Weyl-squared action, $S = \tfrac{1}{4} \int_M |W|^2 dV_g$. From diffeomorphism invariance, the variation of this action $\delta S$ with respect to a variation of the form $\delta g_{ab} = 2 \nabla_{(a} k_{b)}$ vanishes.
Recall from Proposition~\ref{Weyl-squared} that we have an identity for the variation of this action:
\begin{align*}
\delta S = &-\int_M B^{ab} h_{ab} dV_g \\&+\int_{\partial M} \big[\tfrac{4}{3} W_n{}^{bc}{}_n \IIo_{bc} \mathring{h}_{nn} -(\IVo^{ab} + \IIo_c^{(a} W_{n}{}^{b)_{\circ} c}{}_n) \mathring{h}_{ab}^{\otop} - W_n{}^{bc}{}_{n} \delta_{\otop}^{(1)} \mathring{h}_{bc} \big] dV_{\bar g}\,.
\end{align*}
We now specialize to the case where our original manifold is Bach-flat, the boundary is umbilic, and the variation arises from the pullback of an infinitesimal diffeomorphism. Then,
\begin{align*}
\int_{\partial M} \big[ -\IVo^{ab} (\nabla_{(a} k_{b)_{\circ}})^{\otop} - W_n{}^{ab}{}_{n} \delta_{\otop}^{(1)} (\nabla_{(a} k_{b)_{\circ}}) \big] dV_{\bar g} = 0 \,.
\end{align*}
Now observe that
\begin{align*}
\delta_R (\nabla_{(a} k_{b)_{\circ}}) =& \otop [\nabla_n \nabla_{(a} k_{b)_{\circ}} - 2\bar{\nabla}_{(a} (\bar{g}_{b)}^c n^d \nabla_{(c} k_{d)})] \\
=& \otop [\nabla^\top_{(a} (\nabla_n k_{b)_{\circ}})^\top + R_{n(ab)_{\circ}d} k^d - \bar{g}^c_{(a} H\nabla_{|c|} k_{b)_{\circ}} - 2\bar{\nabla}_{(a} (\bar{g}_{b)}^c n^d \nabla_{(c} k_{d)})] \\
=& \otop [\nabla^\top_{(a} (\nabla_n k_{b)_{\circ}})^\top + R_{n(ab)_{\circ}d} k^d  - H\nabla_{(a}^\top k_{b)_{\circ}} - 2\bar{\nabla}_{(a} (\bar{g}_{b)}^c n^d \nabla_{(c}^\top  k_{d)})  \\&- 2\bar{\nabla}_{(a} (\bar{g}_{b)}^c n^d n_{(c} \nabla_{n}  k_{d)})] \\
=& \otop [\nabla^\top_{(a} (\nabla_n k_{b)_{\circ}})^\top + R_{n(ab)_{\circ}d} k^d  - H\bar{\nabla}_{(a} k^\top_{b)_{\circ}}   - \bar{\nabla}_{(a} (n^d \nabla_{b)}^\top  k_{d})  - \bar{\nabla}_{(a} (\bar{g}_{b)}^c  \nabla_{n}  k_{c})] \\
=& \otop [ R_{n(ab)_{\circ}d} k^d - H \bar{\nabla}_{(a} k^\top_{b)_{\circ}} - \bar{\nabla}_{(a} (n^d \nabla_{b)}^\top  k_{d})] \\
=& \otop [ R_{n(ab)_{\circ}d} k^d  - H \bar{\nabla}_{(a} k^\top_{b)_{\circ}}- \bar{\nabla}_{(a} \bar{\nabla}_{b)}  k_n + \bar{\nabla}_{(a} H k^\top_{b)}] \\
=& \otop [ W_{n(ab)n} k_n - (\mathring{\bar{P}}_{ab} + \Fo_{ab}) k_n  - \bar{\nabla}_{a} \bar{\nabla}_{b}  k_n ]\,.
\end{align*}
In the last line above, we used the trace-free part of the Gauss--Fialkow Equation~(\ref{fialkow-gauss}) and the trace-free part of the Codazzi--Mainardi Equation~(\ref{trfr-Codazzi}).
Furthermore, when $\IIo = 0$ and $d=4$, from the trace-free part of the Gauss--Fialkow Equation~(\ref{fialkow-gauss}) we have that $W_{nabn} = -\Fo_{ab}$, so we can integrate by parts to obtain 
\begin{align*}
\int_{\partial M} \big[( \bar{\nabla} \csdot \IVo^{b}) k^\top_b - 2|\Fo|^2 k_n - \mathring{\bar{P}} \cdot \Fo k_n - (\bar{\nabla} \csdot \bar{\nabla} \csdot \Fo) k_n\big] dV_{\bar g} = 0 \,.
\end{align*}
When $\IIo = 0$, we can compute:
\begin{align*}
B_{nn} \=& n^b n^c \nabla^{\top \, a} C_{abc} + |\Fo|^2 + \bar{P} \cdot \Fo \\
\=& -n^b \nabla^{\top \, a} C_{nab} + |\Fo|^2 + \bar{P} \cdot \Fo \\
\=& -n^b \nabla^{\top \, a} C_{nab}^\top + \nabla^{\top \, a} C_{ann} + |\Fo|^2 + \bar{P} \cdot \Fo\\
\=&  \bar{\nabla} \cdot \bar{\nabla} \cdot \Fo + |\Fo|^2 + \bar{P} \cdot \Fo\,,
\end{align*}
where the first equality follows from the fact that $B_{ab} = \nabla^c C_{cab} + P^{cd} W_{adbc}$ and the last equality follows from the identity
$$C_{ann} \= -\tfrac{1}{d-2} \IIo_{ab} \bar{\nabla} \csdot \IIo^b + \bar{\nabla} \csdot \Fo_a - \tfrac{1}{2(d-1)} \bar{\nabla}_a |\IIo|^2\,.$$
Then, we have that
$$B_{nn} - |\Fo|^2 - \mathring{\bar{P}} \csdot \Fo - \bar{\nabla} \csdot \bar{\nabla} \csdot \Fo \= 0\,,$$
and so in the Bach-flat case, the integral identity in question reduces to
\begin{align} \label{divIVo-Fo2-term}
\int_{\partial M} \big[ ( \bar{\nabla} \csdot \IVo^{b}) k^\top_b - |\Fo|^2 k_n \big] dV_{\bar g} = 0 \,.
\end{align}

Recall that the above identity holds for any $k \in \Gamma(TM)$, and so in particular we may choose $k|_{\Sigma} = n$; this implies that
$$\int_{\partial M} |\Fo|^2 dV_{\bar{g}} = 0\,.$$
So it follows that $\Fo = 0$. Alternatively, we may choose $k|_{\Sigma} \in \Gamma(T \Sigma)$ arbitrarily in Equation~(\ref{divIVo-Fo2-term}), finding that $\bar{\nabla} \csdot \IVo = 0$. This establishes that the conformal class metric admits an asymptotically Poincar\'e--Einstein metric in the interior. 

Observe that because $\IIo = \IIIo = 0$, one may easily establish (according to~\cite{GGHW}) that the obstruction density $\mathcal{B}_{\sigma}$ of Proposition~\ref{singyam} vanishes, and hence the singular Yamabe density is smooth, following from Proposition~\ref{unique-singyam}, and hence so is the singular Yamabe metric. Thus, from Corollary~\ref{unique-E}, we have that there exists a unique almost Einstein tensor $E_{ab}$ determined by the boundary embedding and the conformal structure. Note that the transverse jets of $E$ on $\Sigma$ capture the higher conformal fundamental forms~\cite{BGW1}.

Now, to show that $(M^d,\cc)$ admits a (formal to all orders) Poincar\'e--Einstein metric in the conformal class of its interior, it is sufficient to show that the almost Einstein tensor $E_{ab}$ above vanishes to all orders~\cite{GoverAE}---that is, that $\nabla_n^k E_{ab}|_{\Sigma} = 0$ for every $k \in \mathbb{Z}_{\geq 0}$.
As the formal vanishing of $E_{ab}$ is a conformally-invariant condition, it is sufficient to check that $\nabla_n^k E_{ab}|_{\Sigma} = 0$ in some choice of metric representative $g \in \cc$. To fix the metric $g$, we pick any metric $\bar{g} \in \bar{\cc}$ and then extend it to $g \in \cc$ in such a way (according to~\cite[Proposition 7.15]{GPt}) that each $T$-curvature vanishes---that is, the mean curvature $H^g$ vanishes, and all higher transverse order analogs also vanish.


Having established already that $\IIo = \IIIo = 0$, we already have that $\nabla_n^k E \= 0$ for $k = 0,1$. By explicit calculation (see~\cite[Proof of Corollary 4.9]{BGW1}), we see that $\otop \nabla_n^2 E_{ab}|_{\Sigma}$ only involves at most one normal derivative of the metric, and similarly $n^a \nabla_n^2 E_{ab} \= 0$, as may be seen from~\cite[Proof of Lemma 4.11]{BGW1}. As the $T$-curvatures vanish, it follows from the classification result of~\cite{blitz-classifying} (namely, that Riemannian invariants are expressible in terms of intrinsic curvatures, conformal fundamental forms, and $T$-curvatures) that $\nabla_n^2 E|_{\Sigma}$ is expressible in terms of (intrinsic) curvatures of $(\Sigma,\bar{g})$. However, $\nabla_n^2 E|_{\Sigma}$ has weight  $-1$, and there exist no non-vanishing sections expressible in terms of curvatures of $(\Sigma,\bar{g})$ with odd weight. Thus, $\nabla_n^2 E|_{\Sigma} = 0$.

\medskip

Next, consider the case where $k = 3$. 
Because $\sigma$ solves the singular Yamabe problem exactly, we may apply the Thomas-$D$ operator to the equation $\tfrac{1}{d^2} |D \sigma|^2 = 1$ to obtain (via the Leibniz failure lemma of~\cite{will1})
\begin{align*}
\tfrac{1}{d^2} (D^B \sigma)(D_A D_B \sigma) &= X_A |E|^2\,.
\end{align*}
Evaluating the $X$ slot of the above equation and noting that we have already established that $E = \mathcal{O}(\sigma^3)$, we therefore have that
$$\sigma (\nabla \cdot \nabla \cdot E + (d-1) P \cdot E) = \mathcal{O}(\sigma^6)\,,$$
or
$$\sigma \nabla \cdot \nabla \cdot E = - (d-1) \sigma P \cdot E + \mathcal{O}(\sigma^6) = \mathcal{O}(\sigma^4)\,.$$
Hence,
$$\nabla \cdot \nabla \cdot E = \mathcal{O}(\sigma^3)\,.$$
As $\nabla \cdot \nabla \cdot E = n^a n^b \nabla_n^2 E_{ab} + \text{ltots}$, and the lower transverse order terms vanish to higher order in $\sigma$, it follows that
$$n^a n^b \nabla_n^3 E_{ab} \= 0\,.$$

Next, we consider the irreducible component $\otop \nabla_n^3 E$. Observe from~\cite[Lemma 4.7]{BGW1}, the operator $\delta_{4,1}^{(3)} : \Gamma(\odot^2_{\circ} T^* M[1]) \rightarrow \Gamma(\odot^2_{\circ} T^* \Sigma[-2])$ is well-defined and has leading (transverse) symbol $\otop \nabla_n^3$. Thus, from the classification offered in~\cite{blitz-classifying}, it follows that $$\delta_{4,1}^{(3)} E \= \Vo + \text{curvatures invariants of } (\Sigma,\bar{g}) + \text{terms dependent on } \IVo\,.$$
 But note that $\Vo$ vanishes on 4-dimensional Bach-flat manifolds, and there are no non-vanishing sections expressible in terms of curvatures of $(\Sigma,\bar{g})$ of $\odot^2_{\circ} T^* \Sigma[-2]$ on a 3-manifold. Finally, as $\IVo$ has weight $-1$, there are no (non-vanishing) sections of $\odot^2_{\circ} T^* \Sigma[-2]$ that are constructible from $\IVo$ and solely intrinsic geometric quantities---so this term may not contribute. But then it follows that
$$0 \= \otop \nabla_n^3 E + \operatorname{ltots}(E)\,,$$
where $\operatorname{ltots}$ are lower transverse order differential operators, which vanish by the above computations, hence $\otop \nabla_n^3 E = 0$.

Finally we consider the component $(n^a \nabla_n^3 E_{ab})^\top$. Again from the classification result of~\cite{blitz-classifying}, the tensor in question must be expressible in terms of curvatures intrinsic to $(\Sigma,\bar{g})$ and $\IVo$ (as all other quantities appearing in the classification result vanish). It is easy to establish that the only natural sections of $T^* \Sigma[-3]$ that depend on $\IVo$ are $\bar{\nabla} \cdot \IVo$ and $n^a \IVo_{ab}$. However, both of these invariants vanish---the former by the above proof of the manifold admitting an asymptotically Poincar\'e--Einstein metric in its interior and the latter by definition. Thus, $(n^a \nabla_n^3 E_{ab})^\top$ cannot depend on $\IVo$, and hence is expressible entirely in terms of curvature invariants of $(\Sigma,\bar{g})$. But no such non-vanishing Riemannian invariants exist, and so it follows that $(n^a \nabla_n^3 E_{ab})^\top = 0$.

From the considerations above, we have thus established that, in the metric representative where $T_! = T_2 = T_3 = T_4 = 0$, we have that $\nabla_n^3 E \= 0$.

\medskip

It now remains to show that $\nabla_n^k E \= 0$ for every $k \geq 4$. We do so via induction using $k=3$ as the base case. Suppose that $\nabla_n^{\ell}E \= 0$ for all $\ell < k$. We would now like to examine $\nabla_n^k E|_{\Sigma}$. From~\cite[Proof of Lemma 4.11]{BGW1}, we have that since $k \neq d-1 = 3$,
$$n^a \nabla_n^k E_{ab} \= 0\,.$$
Thus we must only examine $\otop \nabla_n^k E|_{\Sigma}$. In what follows, we will implicitly set to zero those terms that must vanish by the induction, and $I^A := \tfrac{1}{d} D^A \sigma$, where $\sigma$ is the unique singular Yamabe density determined above. For a choice of metric representative $g \in \cc$ such that $\sigma = [g; s]$, we have that
\begin{align*}
\otop \nabla_n^k E \=& \otop \nabla_n^k Z^C_{(a} \nabla_{b)_{\circ}} I_C \\
\=& \otop \nabla_n^{k-1} \left[ (\nabla_n Z_{(a}^C) \nabla_{b)_{\circ}} I_C + Z_{(a}^C \nabla_n \nabla_{b)_{\circ}} I_C \right] \\
\=& \otop \nabla_n^{k-1} \left[ Z_{(a}^C \nabla_{b)_{\circ}} \nabla_n I_C + Z^C_{(a} \Omega_{nb)_{\circ} C}{}^D I_D \right] \\
\=& \otop \nabla_n^{k-1} \left[ \nabla_{(a} \nabla_n \nabla_{b)_{\circ}} s + Z_{(a}^C \Omega_{nb)_{\circ} C}{}^D I_D \right] \\ 
\=& \otop \nabla_n^{k-1} \left[ -2\nabla_{(a} \nabla_{b)_{\circ}}  s \rho + Z_{(a}^C \Omega_{nb)_{\circ} C}{}^D I_D \right] \\ 
\=& -\otop \nabla_n^{k-1} \left[ Z_{(a}^C \Omega_{b)_{\circ}n C}{}^D I_D \right] \\
\=& -\otop \nabla_n^{k-2} \nabla^d \left[ Z_{(a}^C \Omega_{b)_{\circ}d C}{}^D I_D \right] \\
\=& \otop \nabla_n^{k-2} \Omega_{d (a| C}{}^D \nabla^d  Z_{b)_{\circ}}^C  I_D \\
\=& 0\,.
\end{align*}
The first identity follows from the fact that $Z^C_{(a} \nabla_{b)_{\circ}} I_C = E_{ab}$~\cite{will1}. The third identity follows by noting that $Z^C \nabla_n I_C$ has lower transverse order than $k+1$ and thus vanishes by the induction hypothesis. The fourth identity follows covariant derivatives of $Z$ and the induction hypothesis, and the fifth identity also follows from the induction hypothesis and the observation~\cite{will1} that $\nabla_n s = 1 - 2 s \rho$, where $\rho = -\tfrac{1}{d}(\Delta s + Js)$. The sixth identity then follows from observing that $s$ must be differentiated, yielding lower order terms which vanish by the induction hypothesis. The seventh identity follows from the identity $n \nabla_n \= \nabla - \nabla^\top$, and the eigth identity follows from a critical result which holds on a 4-dimensional Bach-flat manifold~\cite{GoverSombergSoucek}: that $\nabla \cdot \Omega = 0$. It thus follows that $\nabla_n^k E \= 0$ for all $k$, completing the proof.

%
%
%
%
%
\end{proof}

\begin{remark}
Observe that, in the above proof, we could have terminated our efforts once we had established that the singular Yamabe metric is asymptotically Poincar\'e--Einstein and instead invoked the result of~\cite{AlaeeWool}. To do so would have required establishing their higher-order constraint alluded to above.
\end{remark}

\begin{remark}
Note that Theorem~\ref{umbilic-PE} is equivalent to the statement that, when there exists a unique solution to the boundary-value Bach-flat problem on a given 4-manifold with umbilic boundary, its Dirichlet-to-Neumann map in Display~(\ref{pair-mapping}) becomes
$$(\bar{g}_{ab}, 0) \mapsto (0, -\operatorname{DN}^{(4)})\,,$$
where $\operatorname{DN}^{(4)}$ is the image of the Dirichlet-to-Neumann map for the Poincar\'e--Einstein problem described in~\cite{BGKW}.
\end{remark}

As a corollary, additional hypotheses ensure that the formal Poincar\'e--Einstein structure fixed by the above theorem is in fact a genuine Poincar\'e--Einstein structure in a collar neighborhood of the boundary.
\begin{corollary}
Let $(M^4,\cc)$ be an analytic compact Bach-flat conformal manifold with umbilic boundary $\Sigma$. Then there exists a Poincar\'e--Einstein metric in the conformal class of the interior of a collar neighborhood of $\Sigma$.
\end{corollary}
\begin{proof}
As $(M^4,\cc)$ has an umbilic boundary $\Sigma$ and is Bach-flat, it follows from Theorem~\ref{umbilic-PE} that the almost Einstein tensor vanishes formally to all orders, i.e. $E = \mathcal{O}(\sigma^\infty)$. As the conformal manifold is analytic, it follows that the induced conformal class of metrics on the boundary is also analytic, as is any extrinsic data (including the fourth fundamental form $\IVo$). From~\cite{anderson2010}, it follows that when the boundary data (the induced metric and the image of the Dirichlet-to-Neumann map $g^{(3)}$) is analytic, there exists a unique Poincar\'e--Einstein metric in the interior of a collar neighborhood of the boundary. But the image of the Dirichlet-to-Neumann map $\operatorname{DN}^{(4)}$ is $-\IVo$, which is analytic. The corollary follows.

%
%
\end{proof}

\medskip

It is well-known and easily verified (see e.g.~\cite{FG}) that on an Einstein 4-manifold, the Bach tensor vanishes everywhere. However, a simple counting shows that for most Bach-flat conformal 4-manifolds manifolds with boundary, these manifolds are not Poincar\'e--Einstein.
Having established Theorem~\ref{umbilic-PE}, we now understand why that is the case: a Bach-flat manifold with boundary need only have an umbilic boundary to ensure that it is (at least asymptotically) Poincar\'e--Einstein. While this is still quite a strong condition, it is far weaker than might have been observed naively, namely that $\IIo = \IIIo = 0$~\cite{BGW1} and that $\bar{\nabla} \cdot \IVo = 0$. Nonetheless, we can provide a method of constructing many such manifolds.
%
\begin{corollary}\label{einstein-slicing}
Let $(M^4, g)$ be a closed Einstein manifold and let $\iota : \Sigma \hookrightarrow M$ be an embedded separating hypersurface such that $M = M_+ \sqcup \Sigma \sqcup M_-$. Then, there exists a unique metric $g_+$ in the conformal class of $g|_{M_+}$ such that $(M_+, g^+)$ is asymptotically hyperbolic, has constant (negative) scalar curvature, and is Bach-flat. Furthermore, $\IIo = 0$ if and only if $g^+$ is, formally to all orders, Poincar\'e--Einstein.
\end{corollary}
\begin{proof}
First, observe that for any separating hypersurface embedding, one may always find a unique asymptotically hyperbolic metric $g^+$  in the conformal class of $g|_{M^+}$ such that $(M_+,g^+)$ has constant scalar curvature (see Proposition~\ref{unique-singyam}); this is the singular Yamabe solution $(M_+,g^+)$. Now, because $g$ is Einstein, it follows that $g$ is also Bach-flat. And so, as $g^+$ is in the conformal class of $g|_{M_+}$ and the Bach tensor is a conformal invariant in four dimensions, it follows that $g^+$ is Bach-flat as well.

Next, suppose that $\IIo \neq 0$. Then the conformal class of metrics $[g|_{M_+}]$ does not contain any asymptotically hyperbolic Poincar\'e--Einstein metrics~\cite{GoverAE,LeBrun}. Finally, we have from Theorem~\ref{umbilic-PE} that when $\IIo = 0$, there exists a unique (formally to all orders) Poincar\'e--Einstein metric $g_+$ in $(M_+,[g|_{M_+}])$. But as $g_+$ is (formally) Poincar\'e--Einstein, it is also asymptotically hyperbolic with (asymptotically) constant negative scalar curvature. But by the uniqueness of $g^+$ it follows that $g_+$ must agree with $g^+$. Thus, $g^+$ is (formally to all orders) Poincar\'e--Einstein.
\end{proof}
Umbilic hypersurfaces are generically quite rare. This explains the observation above that most Bach-flat conformal manifolds with boundary do not admit a Poincar\'e--Einstein metric in its conformal class.

\medskip

Theorem~\ref{umbilic-PE} can be extended to the general Fefferman--Graham flat problem in a natural way, by demanding that $\IIo = \IIIo = \cdots = \FF{d/2} = 0$. Asymptotics of the $d=6$ case has been examined by~\cite{anast}. This leads naturally to the following conjecture:
\begin{conjecture}
Let $(M^d,\cc)$ with $d$ even be a Fefferman--Graham-flat conformal manifold whose boundary satisfies $\IIo = \IIIo = \cdots = \FF{d/2} = 0$. Then there exists a (formally to all orders) Poincar\'e--Einstein metric representative in $(M_+,\cc)$.
\end{conjecture}
Proving such a conjecture in general is significantly more challenging than for Theorem~\ref{umbilic-PE}, as that theorem relied on an explicit identity for the Bach tensor---we suspect that similar identities hold for higher-dimensional Fefferman--Graham obstruction tensors, but a general result is currently out of reach.

\bibliographystyle{plain}
\bibliography{biblio}

\end{document}